\documentclass[brochure,english,12pt]{smfbourbaki}

\usepackage[T1]{fontenc}
\usepackage{lmodern,amssymb,bm}

\usepackage{babel}

\usepackage[utf8]{inputenc}

\usepackage[colorlinks=true, linkcolor=blue, citecolor=red,
urlcolor=blue]{hyperref}

\usepackage[
backend=biber,
style=authoryear, 
citestyle=authoryear-comp,
maxnames=7,
sortcites=false, 
uniquename=false 
]{biblatex}
\usepackage{csquotes}

\newtoggle{tnbcbx@howcited}
\DeclareEntryOption[boolean]{howcited}[true]{\settoggle{tnbcbx@howcited}{#1}}
\DeclareBibliographyOption[boolean]{howcited}[true]{\settoggle{tnbcbx@howcited}{#1}}
\DeclareTypeOption[boolean]{howcited}[true]{\settoggle{tnbcbx@howcited}{#1}}

\newbibmacro{howcited}{%
  \iftoggle{tnbcbx@howcited}
    {\iffieldundef{shorthand}
       {}
       {\setunit{\addspace}%
        \printtext[parens]{%
          \bibstring{citedas}%
          \setunit{\addcolon\space}%
          \printfield{shorthand}%
          \setunit{\addslash}%
          }}}
    {}}

\renewbibmacro{finentry}{\usebibmacro{howcited}\finentry}

\DefineBibliographyStrings{french}{citedas    = {cité ci-dessus avec l'acronyme}}
\DefineBibliographyStrings{english}{citedas    = {heretofore cited as}}

\DefineBibliographyExtras{french}{\restorecommand\mkbibnamefamily}

\DeclareDelimFormat{nameyeardelim}{\addcomma\space}

\DeclareNameAlias{sortname}{given-family}

\renewcommand{\bibnamedash}{\leavevmode\raise3pt\hbox to3em{\hrulefill}\space}

\AtEveryBibitem{%
  \clearfield{issn} 
  \clearfield{isbn} 
  \clearfield{doi} 
  \clearlist{language} 
  \ifentrytype{online}{}{
  \ifentrytype{unpublished}{}{
    \clearfield{url}
  }
  }
}

\renewbibmacro{in:}{%
    \ifentrytype{article}{}{\printtext{\bibstring{in}\intitlepunct}}}

\DeclareFieldFormat[article,periodical,inreference]{number}{\mkbibparens{#1}}
\DeclareFieldFormat[article,periodical,inreference]{volume}{\mkbibbold{#1}}
\renewbibmacro*{volume+number+eid}{%
    \printfield{volume}%
    \setunit*{\addthinspace}
    \printfield{number}%
    \setunit{\addcomma\space}%
    \printfield{eid}}

\DeclareFieldFormat[article,inbook,incollection]{title}{\enquote{#1}\addcomma} 

\date{Juin 2026}
\bbkannee{78\textsuperscript{e} année, 2025--2026}  
\bbknumero{1255}                                      

\title{
Bass notes of random hyperbolic surfaces of large genus
}
\subtitle{after Hide–Magee, Anantharaman–Monk and Hide–Macera–Thomas}

\author{Bram Petri}
\address{Institut de Math\'ematiques de Jussieu--Paris\\ Rive Gauche and  Institut universitaire de \\ France ; Sorbonne Universit\'e and Universit\'e \\ Paris Cit\'e, CNRS, IMJ-PRG, F-75005 Paris,\\ France}
\email{bram.petri@imj-prg.fr}

\usepackage{overpic}
\usepackage{dsfont}

\theoremstyle{definition}

\begin{document}

\maketitle

\tableofcontents

\section*{Introduction}

In the last couple of years, multiple constructions of hyperbolic surfaces of large area with near optimal \textbf{spectral gaps}, which had long been sought after, have been found. All of these constructions use \textbf{random surfaces}. The goal of this text is both to put these results in context and to explain some of the ideas behind them.

We will focus on two models of random surfaces. The first model we will discuss yielded the first examples of sequences of closed hyperbolic surfaces whose area tends to infinity and whose spectral gap tends to that of the hyperbolic plane, as proven by \textcite{HideMagee}. This construction uses compactifications of random finite degree covers of the thrice punctured sphere. After this, we will describe the series of works by \textcite{AnantharamanMonk1,AnantharamanMonk2,AnantharamanMonk3,AnantharamanMonk4,AnantharamanMonk5}, culminating in the proof of the fact that near optimal spectral gaps are also typical for a surface of large genus chosen at random using the Weil--Petersson measure. The most recent result we will treat is due to \textcite{HideMaceraThomas}. Their work gives an alternative route to near optimal spectral gaps for Weil--Petersson random surfaces and also yields an error term on the spectral gap that is polynomially small as a function of the genus of the surface.

There are many other results related to spectral gaps of random surfaces that will not be discussed in great detail in this text. Some notable omissions are the work by \textcite{CalderonMageeNaud} on spectral gaps of random covers of Schottky surfaces, the original papers on the polynomial method by
\textcite{ChenGarzaVargasTroppvanHandel,ChenGarzaVargasvanHandel} and the fact that random finite degree covers of a closed surface have a near optimal spectral gap, proven by \textcite{MageePudervanHandel}.

This text is organized as follows. We will start by reminding the reader of some of the basics on the (spectral) geometry of hyperbolic surfaces. After this, we will discuss results from graph theory that have inspired many of the recent developments. Then we will survey some of the recent history on all the different models of random hyperbolic surfaces. Of course there are too many connections to draw and some choices have to be made. These choices are mostly based on the personal preferences of the author and he apologizes in advance to the readers who had hoped to read about one of the many related subjects that haven't made it to the text.

The last three sections are devoted to the results by Hide--Magee, Anantharaman--Monk and Hide--Macera--Thomas respectively. It is however impossible to do justice to these three works in this short text. In particular, we will not attempt to give full proofs, but will rather try to sketch some of the main ideas. For details, we refer the reader to the original articles. 

\subsection*{Acknowledgement} The author would like to thank Nicolas Bourbaki for asking him to write about this series of papers. It has made the author appreciate the many new ideas involved in them even more and he hopes that this text conveys some of them to the reader too. He also thanks  \textcite{sagemath}, Figure \ref{plot_plancherel} was produced using \texttt{SageMath}. Finally, he thanks Claire Burrin, Hans Mahnig, Joe Thomas and Nicolas Bourbaki for comments on a previous version of this text.

\section{The geometry and spectra of hyperbolic surfaces}

Hyperbolic surfaces are a classical object of study and many texts have been written about them, like for instance the books by \textcite{Buser_book,Iwaniec,Bergeron_book,Borthwick_hyperbolic}. This section will present a quick reminder and some examples of hyperbolic surfaces. This is not meant to be a full introduction and we refer the interested reader to the aforementioned books for more details. 

\subsection{Definitions}
We start with a definition:

\begin{defi}
A \textbf{hyperbolic surface} is a complete Riemannian $2$-manifold of constant sectional curvature equal to $-1$.
\end{defi}

The most basic example of a hyperbolic surface is the \textbf{hyperbolic plane}. The hyperbolic plane also allows a short definition, namely it is the unique (up to isometry) simply connected hyperbolic surface without boundary. A concrete Riemannian manifold of sectional curvature equal to $-1$ is usually called a \textbf{model} for the hyperbolic plane. The \textbf{upper half plane model}
\[
\mathbb{H}^2 = \left( \left\{ z\in\mathbb{C};\; \mathrm{Im}(z)>0\right\},\; ds^2 = \frac{dx^2+dy^2}{y^2}\right)
\]
is often convenient for computations and the \textbf{disk model}
\[
\mathbb{D}^2 = \left( \left\{ z\in\mathbb{C};\;\lvert z\rvert < 1\right\},\; ds^2 = 4\frac{dx^2+dy^2}{(1-x^2-y^2)^2}\right)
\]
is often convenient for drawings because it looks compact to our Euclidean eyes. Both of these models are conformal. Indeed, because the metric is in the conformal class of the (incomplete) Euclidean metric, the angles we see are also the angles with respect to the hyperbolic metric. The geodesics in $\mathbb{H}^2$ are half circles orthogonal to $\mathbb{R}$ and vertical lines. In $\mathbb{D}^2$, the geodesics are given by half circles orthogonal to the circle and diagonals through the origin. 

Now that we have a definition of the hyperbolic plane, we also obtain an equivalent definition of hyperbolic surfaces. Indeed, a hyperbolic surface without boundary is a complete Riemannian $2$-manifold that is locally isometric to $\mathbb{H}^2$. If we want to allow totally geodesic boundary, we ask that points on the boundary have a neighborhood that is isometric to the boundary of a half plane delimited by a geodesic in $\mathbb{H}^2$.

Another equivalent definition is that a hyperbolic surface without boundary is a surface of the form $\Gamma\backslash\mathbb{H}^2$, where $\Gamma<\mathrm{Isom}(\mathbb{H}^2)$ is a discrete and torsion free subgroup of the isometry group $\mathrm{Isom}(\mathbb{H}^2)$ of $\mathbb{H}^2$. This isometry group can be identified with $\mathrm{PSL}(2,\mathbb{R}) \rtimes\mathbb{Z}/2\mathbb{Z}$. If we denote the generator of the $\mathbb{Z}/2\mathbb{Z}$ factor by $r$ (for reflection), then $\mathrm{Isom}(\mathbb{H}^2)$ acts on $\mathbb{H}^2$ by 
\[
\left[\begin{array}{cc}
a & b \\
c & d
\end{array}\right] \cdot z = \frac{az+b}{cz+d} \quad \text{and} \quad r(z) = -\overline{z} \quad \text{for } \left[\begin{array}{cc}
a & b \\
c & d
\end{array}\right] \in \mathrm{PSL}(2,\mathbb{R}) \text{ and } z\in\mathbb{H}^2.
\]
In particular, the subgroup $\mathrm{PSL}(2,\mathbb{R})<\mathrm{Isom}(\mathbb{H}^2)$ is exactly the group of orientation preserving isometries. We will only consider oriented surfaces in what follows, so these are quotients $\Gamma\backslash\mathbb{H}^2$ with $\Gamma < \mathrm{PSL}(2,\mathbb{R})$ discrete and torsion free. 

If we want to allow totally geodesic boundary, we need to consider quotients of the form $\Gamma\backslash C$, where $\Gamma<\mathrm{Isom}(\mathbb{H}^2)$ is discrete and torsion-free and $C\subset \mathbb{H}^2$ is a closed and convex subset, preserved by $\Gamma$.

By the \textbf{uniformization theorem}, constant curvature metrics, considered up to scaling, correspond one-to-one to Riemann surface structures considered up to biholomorphism. So we can also think of hyperbolic surfaces as Riemann surfaces.

\subsection{Examples of hyperbolic surfaces}

\subsubsection{The Bolza surface and Hurwitz surfaces}

Thus far, our only explicit example of a hyperbolic surface is the hyperbolic plane, so its high time for some more examples. We start with the \textbf{Bolza surface}. The quickest way to define it is to say that is the unique closed Riemann surface of genus $2$ that has exactly $48$ automorphisms (or orientation preserving isometries of the hyperbolic metric). Indeed, in general Hurwitz's theorem implies that the number of automorphisms of a closed Riemann surface of genus $g$ is at most $84(g-1)$, and when $g$ is fixed, there are at most finitely many surfaces realizing this bound. If there are surfaces realizing this bound in a given genus, then they are called \textbf{Hurwitz surfaces}. In genus $2$, there aren't any Hurwitz surfaces and the Bolza surface is the unique maximizer of the size of the automorphism group.

\begin{figure}[ht]
\begin{center}
\begin{overpic}{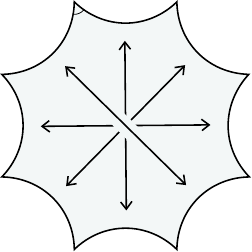}
\put(30,85){$\frac{\pi}{4}$}
\end{overpic}
\caption{A regular octagon and the side pairings that turn it into the Bolza surface}\label{pic_bolza}
\end{center}
\end{figure}

Of course, the Bolza surface also admits a more explicit description. Namely, it's the hyperbolic surface one obtains by gluing together the pairs of opposite sides of a regular hyperbolic octagon (so its sides are geodesic segments that all have the same length and its interior angles are all equal) with interior angles $\frac{\pi}{4}$, as in Figure \ref{pic_bolza}. 

The Bolza surface can also be thought of as a \textbf{triangle surface}: a hyperbolic surface $\Gamma\backslash\mathbb{H}^2$ determined by a normal subgroup $\Gamma \triangleleft \Theta(p,q,r)$ of finite index, where $\Theta(p,q,r)$ is the group generated by the rotations of orders $p$, $q$ and $r$ in the vertices of a hyperbolic triangle with interior angles $\frac{\pi}{p}$, $\frac{\pi}{q}$ and $\frac{\pi}{r}$ respectively (which exists if and only if \linebreak $\frac{1}{p}+\frac{1}{q}+\frac{1}{r}<1$). 

In the case of the Bolza surface, we can take $(p,q,r)=(2,3,8)$ in which case the index is $48$. Indeed, in general $\Theta(p,q,r)/\Gamma$ acts on $\Gamma\backslash\mathbb{H}^2$ by orientation preserving isometries. By work of \textcite{Takeuchi} (see also Section 13.3 of the book by \textcite{MaclachlanReid}), $\Theta(2,3,8)$ is an \textbf{arithmetic group}, so the Bolza surface is also an \textbf{arithmetic surface}. For more on arithmetic groups, including their definition, we refer the reader to the books by \textcite{MaclachlanReid,WitteMorris}.

It turns out that Hurwitz surfaces correspond exactly to normal subgroups of $\Theta(2,3,7)$. In particular, again using Takeuchi's theorem, these are all arithmetic as well. The Hurwitz surface of smallest genus is the \textbf{Klein quartic}, the unique Hurwitz surface of genus $3$, a hyperbolic surface about which a whole book has been written, edited by \textcite{Eightfoldway}.

\subsubsection{The thrice punctured sphere}\label{sec_thrice_punc}

Our next example is a non-compact surface of finite area: the thrice punctured sphere. It follows from the uniformization theorem, combined with the fact the the automorphism group of the Riemann sphere acts triply transitively, that there is only one complete hyperbolic metric on the thrice punctured sphere.

\begin{figure}[ht]
\begin{center}
\begin{overpic}{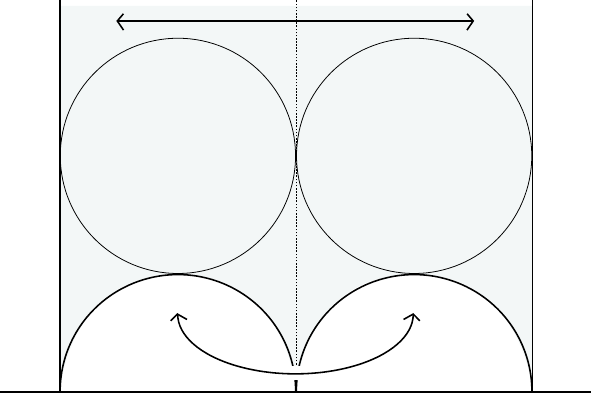}
\put(7,-4){$-1$}
\put(49,-4){$0$}
\put(89,-4){$1$}
\end{overpic}
\caption{Building the thrice punctured sphere}\label{pic_thrice_punctured}
\end{center}
\end{figure}

To build it explicitly, we can take the ideal quadrilateral in $\mathbb{H}^2$ with ideal vertices $-1$, $0$, $1$ and $\infty$ and glue the sides together according to the pattern drawn in Figure \ref{pic_thrice_punctured}. Because of the ideal vertices, the sides have infinite length. So per pair of sides there is a one parameter family of isometries between them that we could use for the gluing and thus there are two choices to be made. This is what the inscribed disks are for. They define $4$ special points on the sides of our quadrilateral. The gluing that guarantees that the resulting hyperbolic metric is complete is the unique gluing by isometries that identifies these special points. 

This is also exactly the gluing that ensures that the corresponding isometries of $\mathbb{H}^2$ are parabolic, so the structure near the punctures is that of a hyperbolic \textbf{cusp}. That is, the punctures have neighborhoods that are isometric to
\[
C_t = \left\{ z\in \mathbb{H}^2;\; \mathrm{Im}(z)>t\right\} \; \Big/ \; \left\langle \left[\begin{array}{cc}
1 & 1 \\
0 & 1
\end{array}\right] \right\rangle
\]
for some $t>0$. 

In fact, the elements of $\mathrm{PSL}(2,\mathbb{R})$ that correspond to the two gluings are given by 
\[
\left[\begin{array}{cc}
1 & 2 \\
0 & 1
\end{array}\right] \quad \text{and} \quad \left[\begin{array}{cc}
1 & 0 \\
2 & 1
\end{array}\right],
\] 
which freely generate the level $2$ \textbf{principal congruence subgroup} of $\mathrm{PSL}(2,\mathbb{Z})$, i.e. the group
\[
\Gamma(N) = \left\{
\left(
\begin{array}{cc}
a & b \\
c & d
\end{array}\right) \in\mathrm{SL}(2,\mathbb{Z});\; \begin{array}{c}
a \equiv d \equiv 1 \mod N \\
b \equiv c \equiv 0 \mod N
\end{array}\right\} \; \Big/\; \left\{ \pm \left(\begin{array}{cc} 1 & 0 \\ 0 & 1 \end{array} \right)\right\}
\]
for $N=2$. In other words the thrice punctured sphere can be uniformized as $\Gamma(2)\backslash\mathbb{H}^2$. Because $\mathrm{PSL}(2,\mathbb{Z})$ is the $(2,3,\infty)$-triangle group, we are still in the realm of arithmetic triangle surfaces.

\subsubsection{Pants decompositions}\label{sec_pants_dec}

As even the reader who is less familiar with hyperbolic surfaces probably suspects, most hyperbolic surfaces are in fact not arithmetic triangle surfaces. Indeed, on a fixed surface of finite type, there are uncountably many pairwise non-isometric hyperbolic structures, whereas only finitely many of these structures are arithmetic. In fact, the factorial growth rate of this number is known, due to \textcite{BelolipetskyGelanderLubotzkyShalev}. 

In this section we will sketch how to build all orientable hyperbolic surfaces of finite area.

\begin{figure}[ht]
\begin{center}
\includegraphics[scale=1]{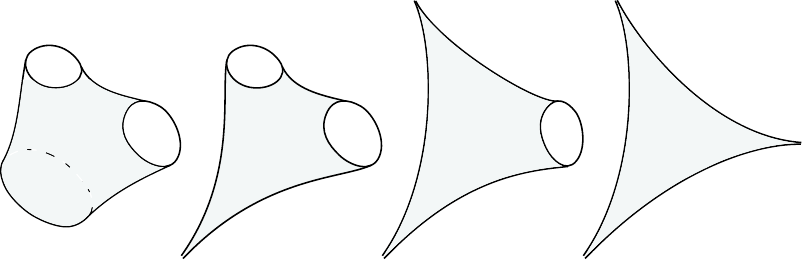}
\caption{Four topological types of pants}\label{pic_pants}
\end{center}
\end{figure}

We start with a building block (see Figure \ref{pic_pants}):
\begin{defi}
A \textbf{pair of pants} is an orientable hyperbolic surface with empty or totally geodesic boundary that is homeomorphic to a sphere with $n$ punctures and $b$ boundary components such that $n+b=3$.
\end{defi}

What makes these convenient building blocks is that, up to isometry, pairs of pants are determined by the lengths of their boundary components (where we will identify a cusp with a boundary component of length $0$). 

We can now build hyperbolic surfaces of finite area by gluing together a finite number of pairs of pants. The only thing to make sure of is that those pairs of boundary components that we plan to glue together have the same length. We can use any isometry between such a pair of boundary components for the gluing, so per pair, we obtain a one parameter family of gluings.

The fact that this construction yields all orientable hyperbolic surfaces of finite area follows from the fact that free homotopy classes of essential\footnote{\label{fn_essential_curve}A closed curve is called essential if it is not homotopic to a point, a puncture or a boundary component.} closed curves contain unique geodesics and that these geodesic representatives also minimize the total number of intersections. So, given a hyperbolic surface, we can always find a \textbf{pants decomposition} -- a collection of simple closed curves that cut it into pairs of pants -- by putting a topological pants decomposition in geodesic position.

\subsection{Teichmüller and moduli spaces}
The above discussion implies that if we fix a pants decomposition of a surface, then we obtain a parameter space of hyperbolic surfaces. Indeed, to every curve in the pants decomposition correspond two parameters: a \textbf{length} and a \textbf{twist} that records the isometry we used for the gluing. 

At this point, it may seem most natural to let the twists live on the circle (corresponding to where some base point lands under the isometry). We will however let the twist be a real number instead. So if $\gamma$ is a curve in the pants decomposition and we change the twist parameter from $\tau_\gamma$ to $\tau_\gamma+\delta$, then the new surface is obtained as follows. We open the old surface up along the geodesic homotopic to $\gamma$ and complete it so as to obtain a surface with boundary components $\gamma_1$ and $\gamma_2$ that we parametrize with unit speed, in such a way that the gluing ``$\gamma_1(t)\sim\gamma_2(t)$ for all $t\in\mathbb{R}$'' corresponds to the original surface. The gluing ``$\gamma_1(t) \sim \gamma_2(t+\delta)$'' then yields the surface we are after. See Figure \ref{pic_twist} for a depiction of this process.

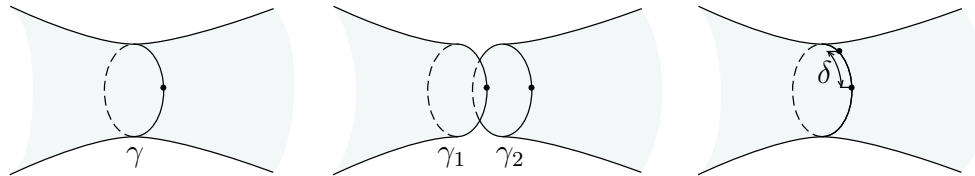
\begin{figure}[ht]
\begin{center}
\begin{overpic}{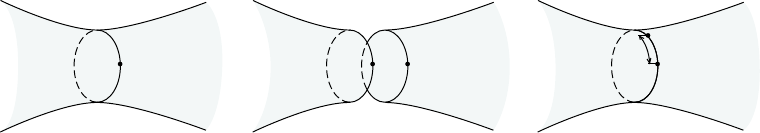}
\put(12,1.5){$\gamma$}
\put(44,1.5){$\gamma_1$}
\put(50,1.5){$\gamma_2$}
\put(83,9.5){$\delta$}
\end{overpic}
\caption{Performing a twist}\label{pic_twist}
\end{center}
\end{figure} 

An Euler characteristic computation implies that a pants decomposition of an orientable surface of genus $g$ with $n$ punctures and $b$ boundary components contains $3g+n+b$ curves. This does not count the curves that form the potential boundary components. In the deformation spaces of hyperbolic surfaces that we will consider here, we will always assume the boundary lengths to be fixed. The space of hyperbolic surfaces we obtain is hence a space homeomorphic to $\Big(\mathbb{R}_{>0}\times\mathbb{R}\Big)^{3g+n+b-3}$. 

This deformation space turns out to be homeomorphic to the \textbf{Teichmüller space} of the underlying surface. As sets, these are defined as:
\[
\mathcal{T}_{g,n}(L_1,\ldots,L_n) = \left\{(X,f);\; \begin{array}{c}
X \text{ a hyperbolic surface with boundary}\\
 \text{components }\delta_1,\ldots,\delta_n \text{ of length }L_1,\ldots,L_n\\
\text{ respectively, and }f:\Sigma_{g,n}(L) \to X \text{ an} \\
\text{orientation preserving diffeomorphism,} \\
\text{with }f(\beta_i)=\delta_i,\; i=1,\ldots,n
\end{array} \right\} \; \Big /\; \sim.
\]
Here $L \in [0,\infty)^n$ and $L_i=0$ corresponds to a cusp and $\Sigma_{g,n}(L)$ is a fixed reference surface with boundary components and punctures $\beta_1,\ldots,\beta_n$: if $L_i=0$ then $\beta_i$ is a puncture and otherwise it's a boundary component. Finally $(X_1,f_1)\sim (X_2,f_2)$ if and only if there exists an isometry $\varphi:X_1\to X_2$ that maps the $i$\textsuperscript{th} boundary component of $X_1$ to the $i$\textsuperscript{th} boundary component of $X_2$ for $i=1,\ldots,n$ and such that 
\[
f_2^{-1}\circ \varphi\circ f_1 : \Sigma_{g,n}(L) \to \Sigma_{g,n}(L)
\] 
is homotopic to the identity. We will write $\Sigma_g := \Sigma_{g,0}$ and $\mathcal{T}_g :=\mathcal{T}_{g,0}$.

The identification with $\Big(\mathbb{R}_{>0}\times\mathbb{R}\Big)^{3g+n-3}$ now goes as follows. We first fix a pants decomposition of $\Sigma_{g,n}(L)$, i.e. a collection of curves $\alpha_1,\ldots,\alpha_{3g+n-3}$ that decompose $\Sigma_{g,n}(L)$ into pairs of pants. Then given a point $[X,f]\in\mathcal{T}_{g,n}(L)$, we record the lengths and twists of $f(\alpha_i)$, giving us a point in  $\Big(\mathbb{R}_{>0}\times\mathbb{R}\Big)^{3g+n-3}$. The lengths are easy to define, $X$ carries a hyperbolic metric, so we can put $f(\alpha_i)$ in geodesic position. The twists are more subtle and we refer to the literature for a proper definition. The coordinates we have just described are called \textbf{Fenchel--Nielsen coordinates}. There is an independently defined topology on Teichmüller space, which turn these coordinates into a homeomorphism. We will not get into this in this text and just say that the topology on $\mathcal{T}_{g,n}(L)$ is given by these coordinates.

We have already observed that Teichmüller space overparametrizes hyperbolic surfaces. In fact, taking the twist in $\mathbb{R}$ instead of the circle is not the only source of this. Namely, the \textbf{mapping class group} $\mathrm{MCG}(\Sigma_{g,n}(L))$ of $\Sigma_{g,n}(L)$ -- the group of homotopy classes of orientation preserving diffeomorphisms that restrict to the identity on $\partial\Sigma_{g,n}(L)$ and also don't permute the punctures -- acts on $\mathcal{T}_{g,n}(L)$ by
\[
[\varphi]\cdot [X,f] = [X,f\circ\varphi^{-1}],\quad [X,f]\in\mathcal{T}_{g,n}(L),\;[\varphi] \in \mathrm{MCG}(\Sigma_{g,n}(L)).
\]
The corresponding \textbf{moduli space} is the quotient of this action:
\[
\mathcal{M}_{g,n}(L) = \mathrm{MCG}(\Sigma_{g,n}(L)) \; \backslash \; \mathcal{T}_{g,n}(L).
\]
We note that punctures and boundary components are still distinguishable in $\mathcal{M}_{g,n}(L)$, because we have chosen to quotient only by mapping classes that preserve them. We will write $\mathcal{M}_g := \mathcal{M}_{g,0}$.

The mapping class group acts properly discontinuously, but not freely, on $\mathcal{T}_{g,n}(L)$, so $\mathcal{M}_{g,n}(L)$ has the structure of a $(6g+2n-6)$-dimensional orbifold. When either $n=0$ or $L=(0,\ldots,0)$ this space coincides with the moduli space of Riemann surfaces (potentially with marked points), considered up to biholomorphism. For an introduction to this subject, we refer the reader to the books by \textcite{ImayoshiTaniguchi,FarbMargalit}.

\subsection{Spectral theory}

Next up, we discuss the \textbf{Laplacian} operator on a hyperbolic surface. On a general Riemannian manifold $M$, this is the composition of minus the divergence with the gradient $\Delta_M=-\mathrm{div}\circ\mathrm{grad}:C^\infty(M)\to C^\infty(M)$. If $M$ is geodesically complete (as all manifolds without boundary that we will consider are), then $\Delta_M$ is an essentially self-adjoint operator with domain $C^\infty_c(M)$: the space of compactly supported smooth functions. 

For the spectral theory on general Riemannian manifolds we refer the reader to the books by \textcite{Chavel,Borthwick_spectral_theory}. We will now give a brief introduction to the spectral theory of hyperbolic surfaces.

\subsubsection{The hyperbolic plane}
We start with the hyperbolic plane. A computation shows that in the upper half plane model,
\[
\Delta_{\mathbb{H}^2} = - y^2 \cdot \left(\frac{\partial^2}{\partial x^2} + \frac{\partial^2}{\partial y^2}\right)
\]
and in the disk model
\[
\Delta_{\mathbb{D}^2} = -\frac{(1-x^2-y^2)^2}{4}\cdot \left(\frac{\partial^2}{\partial x^2} + \frac{\partial^2}{\partial y^2}\right).
\]
The presence of the inverse of the conformal factor in this expression is a $2$-dimensional phenomenon.

The spectrum of $\Delta_{\mathbb{H}^2}$ is given by the following lemma:
\begin{lemm}\label{lem_spec_H2}
The spectrum of $\Delta_{\mathbb{H}^2}$ on $L^2(\mathbb{H}^2)$ is $[\frac{1}{4},\infty)$.
\end{lemm}

\begin{proof}[Proof sketch]
First we prove that the spectrum of $\Delta_{\mathbb{H}^2}$ is contained in $[\frac{1}{4},\infty)$. For this part we use an argument that can for instance be found in the article by \textcite[page 359]{McKean}. Take any $f\in C_c^\infty(\mathbb{H}^2)$. For $x\in\mathbb{R}$ fixed, integration by parts gives
\begin{align*}
\frac{1}{4}\left(\int_0^\infty \frac{f^2}{y^2}\;dy\right)^2 & = \frac{1}{4} \left(\int_0^\infty \frac{\partial(f^2)}{\partial y}\cdot \frac{dy}{y}\right)^2\\
 & = \left( \int_0^\infty \frac{f\cdot f'}{y} \; dy \right)^2 \leq \int_0^\infty \frac{f^2}{y^2}\;dy \cdot \int_0^\infty (f')^2\;dy.
\end{align*}
This implies that
\[
\frac{1}{4}\int_0^\infty \frac{f^2}{y^2}\; dy \leq \int_0^\infty (f')^2\; dy = -\int_0^\infty \frac{f}{y^2} \cdot y^2 \frac{\partial^2f}{\partial y^2}\; dy,
\]
integrating by parts again in the last equality. Now we integrate with respect to $x$ and obtain
\[
\frac{1}{4}\int_{\mathbb{H}^2} \frac{f^2}{y^2}\; dx\;dy \leq - \int_{\mathbb{H}^2}\frac{f}{y^2}\cdot y^2 \frac{\partial^2f}{\partial y^2}\;dxdy \leq - \int_{\mathbb{H}^2}\frac{f}{y^2}\cdot  \Delta_{\mathbb{H}^2}f\;dxdy, 
\]
where we have used the fact that for $y$ fixed $-\int_{-\infty}^\infty f \cdot \frac{\partial^2 f}{\partial x^2}\; dx \geq 0$ in the last inequality. Using the variational characterization of the spectrum, this proves that $\Delta_{\mathbb{H}^2} \subset [\frac{1}{4},\infty)$.

The eigenfunctions of $\Delta_{\mathbb{H}^2}$ are not in $L^2(\mathbb{H}^2)$, so one way to prove that the interval $[\frac{1}{4},\infty)$ is contained in the spectrum of $\Delta_{\mathbb{H}^2}$, is to construct, for all $\lambda\in [\frac{1}{4},\infty)$, sequences of functions $\phi_n^{(\lambda)}\in L^2(\mathbb{H}^2)$ such that $\lVert (\Delta_{\mathbb{H}^2}-\lambda)\phi_n^{(\lambda)} \rVert_2 \to 0$ as $n\to\infty$. McKean uses approximations of conical functions. Another standard choice is to approximate the function $y\mapsto y^s$, where $\lambda = s(1-s)$, by multiplying it with a sequence of appropriate cut-off functions, see for instance the proof of Proposition 7.2 in the book by \linebreak \textcite{Borthwick_hyperbolic}. 
\end{proof}

There is more to say, but we will move on to the hyperbolic surfaces that this story is about, namely surfaces of finite area, starting with closed surfaces.

\subsubsection{Closed surfaces}

We start with the spectral theorem in the closed setting. In this text, the definition of a \textbf{closed} manifold will include the requirement that the manifold is connected, on top of being compact and not having boundary.
\begin{theo}[Spectral theorem] Let $X$ be a closed hyperbolic surface. Then the spectrum of $\Delta_X$ on $L^2(X)$ is discrete and consists entirely of eigenvalues
\[
0=\lambda_0(X) \quad < \quad \lambda_1(X) \quad \leq \quad  \lambda_2(X) \quad \leq  \quad \ldots .
\]
The corresponding eigenfunctions are smooth and an orthornomal basis of $L^2(X)$ can be extracted from the set of eigenfunctions.
\end{theo}

In the above, the zeroth eigenvalue corresponds to constant functions. Because our surface is assumed to be connected, the next eigenvalue $\lambda_1(X)$ is strictly positive. In fact, this theorem has nothing to do with hyperbolic surfaces and holds for closed Riemannian manifolds in general. Likewise, on all closed Riemannian manifolds, the number of eigenvalues below a given threshold grows according to \textbf{Weyl's law}, which on a closed hyperbolic surface states that
\[
N(\lambda) = \#\left\{k;\; \lambda_k(X) \leq \lambda\right\} \stackrel{\lambda\to\infty}\sim \frac{1}{4\pi}\cdot\mathrm{area}(X)\cdot \lambda.
\]
The Gau\ss--Bonnet theorem implies that the area of a closed and orientable hyperbolic surface of genus $g$ satisfies $\mathrm{area}(X)=4\pi\cdot (g-1)$.

The eigenvalues $\lambda_i(X)$ that lie in the interval $(0,\frac{1}{4})$ are called the \textbf{small eigenvalues} of $X$. \textcite{OtalRosas} proved that there are at most $2g-3$ small eigenvalues on a hyperbolic surface of genus $g$, confirming a conjecture by \textcite{Buser_small_ev}, who also proved that there are surfaces that realize this number of small eigenvalues.

\subsubsection{Non-compact surfaces of finite area}
The spectral theory of non-compact hyperbolic surfaces is more subtle. We refer the reader to the textbook by \textcite{Iwaniec} for a proper introduction. Here we will just describe the spectral decomposition of $L^2(X)$, without proof.

First, we briefly need to mention Eisenstein series. In the proof sketch of Lemma \ref{lem_spec_H2}, we have already noted that $x+iy \in \mathbb{H}^2\mapsto y^s$ is an eigenfunction of $\Delta_{\mathbb{H}^2}$ that does not lie in $L^2(\mathbb{H}^2)$. If $X=\Gamma\backslash\mathbb{H}^2$ is a non-compact hyperbolic surface of finite area, then we can use this function to build an eigenfunction for $\Delta_X$ as well. To do this, we note that cusps correspond one-to-one to conjugacy classes of maximal parabolic subgroups of $\Gamma$. For each of the cusps $\mathbf{c}$ of $\Gamma$, thought of as a representative in $\partial_\infty\mathbb{H}^2 = \mathbb{R}\cup\{\infty\}$ for a $\Gamma$-orbit of parabolic fixed points, we fix some element $h_{\mathbf{c}}\in\mathrm{PSL}(2,\mathbb{R})$ such that $h_{\mathbf{c}}(\mathbf{c}) = \infty$ and $h_{\mathbf{c}}$ conjugates the maximal parabolic subgroup corresponding to $\mathbf{c}$ to 
\[
\left\langle \left[\begin{array}{cc} 1 & 1 \\ 0 & 1 \end{array}\right]\right\rangle.
\] We will write $\Gamma_{\mathbf{c}}<\Gamma$ for the stabilizer of $\mathbf{c}$ in $\Gamma$. 

If we now fix a complex number $s$ with $\mathrm{Re}(s)>1$, then the \textbf{Eisenstein series}
\[
z\in\mathbb{H}^2\mapsto E_{\mathbf{c}}(z;s) = \sum_{[\gamma] \in \Gamma_{\mathbf{c}}\backslash\Gamma} \mathrm{Im}(h_{\mathbf{c}}\cdot \gamma \cdot z)^s
\]
is a $\Gamma$-invariant eigenfunction of $\Delta_{\mathbb{H}^2}$ of eigenvalue $s(1-s)$. Because of the invariance, it can be thought of as a function on $X$ as well. However, it's not an element of $L^2(X)$. To turn this function into an $L^2(X)$-function (but lose the fact that it's an eigenfunction), we take $\psi\in C_c(0,\infty)$ and define the \textbf{incomplete Eisenstein series}
\[
E_{\mathbf{c}}(z;\psi) = \sum_{[\gamma] \in \Gamma_{\mathbf{c}}\backslash\Gamma} \psi(\mathrm{Im}(h_{\mathbf{c}}\cdot \gamma \cdot z)).
\]
We will write $\mathcal{E}(X)\subset L^2(X)$ for the closed linear subspace spanned by all incomplete Eisenstein series.

In general, $\mathcal{E}(X)$ is a proper subspace of $L^2(X)$. A computation shows that its orthogonal complement is spanned by bounded functions $f\in C^\infty(X)$ whose zeroth Fourier coefficient in every cusp vanishes. Indeed, since $f(h_{\mathbf{c}}^{-1} (z+1)) = f(h_{\mathbf{c}}^{-1} z)$ for all $z\in\mathbb{H}^2$, we can write a Fourier series expansion for $f(h_{\mathbf{c}}^{-1}\;\cdot\;)$ with respect to the horizontal coordinate. The zeroth coefficient in this expansion is
\[
\widehat{f}_{\mathbf{c}}(y) = \int_0^1 f(h_{\mathbf{c}}^{-1}\cdot (x+iy))\; dx.
\]
We will write $\mathcal{C}(X)$ for the $L^2(X)$-closure of the set of bounded smooth functions with $\widehat{f}_{\mathbf{c}}= 0$ for all cusps $\mathbf{c}$ of $X$. We thus have the following orthogonal decomposition
\[
L^2(X) = \mathcal{C}(X) \stackrel{\perp}{\oplus} \mathcal{E}(X).
\]
The Laplacian preserves this decomposition and it turns out that the spectrum on $\mathcal{C}(X)$ is discrete. The corresponding eigenfunctions are called \textbf{cusp forms}.

\textcite{Selberg_Fourier_coeff} proved that Eisenstein series can be meromorphically continued to the whole $s$-plane. The resulting extensions do not have poles on the line $\{\mathrm{Re}(s)=\frac{1}{2}\}$ and the resulting function with parameter $\frac{1}{2}+it$ ($t\in\mathbb{R}$) is a generalized eigenfunction (so not a funtion in $L^2(X)$) with eigenvalue $\frac{1}{4}+t^2$. As a result, the spectrum of $\Delta_X$ on $\mathcal{E}(X)$ includes the ray $[\frac{1}{4},\infty)$ and this is continuous spectrum, i.e. for $\lambda \in [\frac{1}{4},\infty)$, $\left(\Delta_X-\lambda\right)\rvert_{\mathcal{E}(X)}$ is injective and has dense range, but is not surjective.

The final source of spectrum is given by residues of Eisenstein series. In the half plane $\{\mathrm{Re}(s)>\frac{1}{2}\}$, the meromorphically continued Eisenstein series have finitely many poles that all lie in $(\frac{1}{2},1]$. The residues at these poles in $(\frac{1}{2},1]$ are Laplacian eigenfunctions in $\mathcal{E}(X)$. The corresponding spectrum is called \textbf{residual spectrum}. This includes constant functions, that appear as residues at $s=1$. 
 
We summarize the above in a theorem:
\begin{theo}
Let $X$ be a non-compact hyperbolic surface of finite area. Then $L^2(X) = \mathcal{C}(X) \stackrel{\perp}{\oplus}\mathcal{E}(X)$ and this decomposition is preserved by $\Delta_X$. Moreover, 
\begin{itemize}
\item the spectrum of $\Delta_X$ on $\mathcal{C}(X)$ is discrete (and thus consists entirely of eigenvalues)
\item the spectrum of $\Delta_X$ on $\mathcal{E}(X)$ contains the ray $[\frac{1}{4},\infty)$, which is the continuous spectrum of $\Delta_X$, and finitely many eigenvalues in $[0,\frac{1}{4})$.
\end{itemize}
\end{theo}
For proofs, we refer the reader to the book by \textcite[Theorem 4.7 and Theorem 7.3]{Iwaniec}.

Selberg proved, using his trace formula (see the next section), that for congruence groups $\Gamma$ (groups that contain $\Gamma(N)$ for some $N$) the cuspidal spectrum, i.e. the spectrum of $\Delta_{\Gamma\backslash\mathbb{H}^2}$ on $\mathcal{C}(\Gamma\backslash\mathbb{H}^2)$, satisfies a Weyl law. It is however expected that this is specific to arithmetic surfaces and that in most moduli spaces of non-compact surfaces of finite area, a generic surface supports only finitely many cusp forms, see the survey by \textcite[Conjecture 1]{Sarnak_Spectra}.

Again, there is a lot more to say both in the closed case and the finite area case, we for instance haven't discussed the connection with the decomposition of $L^2(\Gamma\backslash G)$ into irreducible representations yet (see the book by \textcite{GelfandGraevPyatetskiiShapiro} for this), but the space here is too limited to go through everything.

\subsubsection{The Selberg trace formula}\label{sec_STF}

One of the ways of accessing the spectrum of a hyperbolic surface of finite area is through the Selberg trace formula. This formula provides an explicit connection between the geometry and the spectrum of a hyperbolic surface and can be thought of as a hyperbolic analogue of the Poisson summation formula (that relates the spectrum and the length spectrum of a flat torus). Because we will only use it for closed surfaces in this text, we will restrict to that case:
\begin{theo}[Selberg trace formula] Let $X$ be a closed hyperbolic surface, $f\in C_c^\infty(\mathbb{R})$ be an even function and $\widehat{f}(\xi) = \int_{-\infty}^\infty e^{-i\cdot \xi\cdot x} f(x)\, dx$ its Fourier transform. Then
\begin{align*}
\sum_{\lambda \in \mathrm{spec}(\Delta_X)} \widehat{f}\left(\sqrt{\lambda - \frac{1}{4}}\right) &=  \frac{\mathrm{area}(X)}{4\pi} \int_{-\infty}^\infty y \widehat{f}(y) \tanh(\pi y)\,dy \\ \notag
  &\quad + \sum_{\substack{\gamma \text{ primitive closed} \\ \text{geodesic on }X }} \ell(\gamma)\; \sum_{n\geq 1} \frac{1}{2\sinh(n\ell(\gamma)/2)} f(n\ell(\gamma)),
  \end{align*}
where $\mathrm{spec}(\Delta_X)$ denotes the spectrum of $\Delta_X$ on $L^2(X)$ and $\ell(\gamma)$ denotes the length of the geodesic $\gamma$.
\end{theo}

The left hand side, involving the spectrum of the Laplacian is usually called the \textbf{spectral side} of the formula and the right hand side is usually called the \textbf{geometric side}. We already see from the spectral side that small eigenvalues play a special role in the theory: they get mapped to the imaginary axis by $\lambda\mapsto \sqrt{\lambda-\frac{1}{4}}$.

For the case of surfaces of finite area, we refer to the books by \textcite{Iwaniec,Hejhal_vol1,Hejhal_vol2}.

\section{The bass note}

We now turn to the spectral invariant that this text is about: the bass note of a hyperbolic surface. If our hyperbolic surface $X$ is closed, then its bass note is simply $\lambda_1(X)$: the smallest non-zero eigenvalue of its Laplacian. If $X$ is non-compact but of finite area, then there is a choice to be made. We could consider the smallest non-zero eigenvalue (if indeed there is a non-zero eigenvalue) or the spectral gap -- the infimum of the positive spectrum of $\Delta_X$ -- instead. We will choose the latter in this text. In particular, if $X$ is of finite area but non-compact, we will abuse notation slightly and write
\[
\lambda_1(X) = \inf\Big(\mathrm{spec}(\Delta_X)-\{0\}\Big).
\]
This choice also implies that (since we're always assuming that $X$ is connected):
\begin{equation}\label{eq_var_char}
\lambda_1(X) = \inf\left\{ \frac{\int_X \lVert \nabla f \rVert^2}{\int_X f^2};\; f\in C^\infty_c(X) \text{ such that }\int_X f =0\right\}.
\end{equation}

\subsection{A measure of connectivity}

This a priori analytical invariant contains a lot of geometric information about the surface. Indeed, it can be seen as a measure of connectivity of the surface $X$. It's for instance related to the \textbf{Cheeger constant} of $X$:
\[
h(X) = \inf\left\{ \frac{\ell(\partial Y)}{\mathrm{area}(Y)};\; \begin{array}{c}
Y\subset X \text{ a subsurface with smooth} \\
\text{boundary and }\mathrm{area}(Y) \leq \mathrm{area}(X)/2
\end{array} \right\}.
\]
This number measures how hard it is to cut the surface in two. Figure \ref{pic_cheeger} shows an example of what a surface with a small Cheeger constant might look like: it contains a big subsurface $Y$ that is connected to the rest of the surface by a small interface $\partial Y$. 

\begin{figure}[ht]
\begin{center}
\begin{overpic}{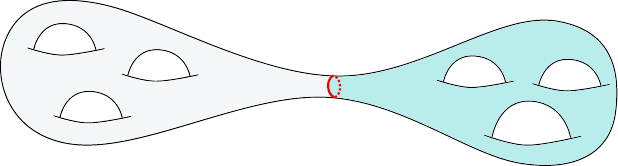}
\put(52,16){$\partial Y$}
\put(73,7.5){$Y$}
\end{overpic}
\caption{A surface with small Cheeger constant}\label{pic_cheeger}
\end{center}
\end{figure}

Its connection to the bass note is given by the following theorem due to \textcite{Cheeger,Buser_note}, which in the case of hyperbolic surfaces reads:
\begin{theo}[Cheeger--Buser inequality]\label{thm_CheegerBuser}
Let $X$ be a hyperbolic surface of finite area. Then
\[
\frac{h(X)^2}{4} \leq \lambda_1(X) \leq 2\;h(X) + 10\;h(X)^2.
\]
\end{theo}

This inequality also implies that, whenever $(g,n)\notin\{(0,3),(1,1)\}$, then
\[
\inf_{X\in\mathcal{M}_{g,n}} \lambda_1(X) = 0. 
\]
Indeed, under our assumption on $(g,n)$, an orientable surface of genus $g$ with $n$ punctures contains a separating simple closed curve $\alpha$ that is not homotopic to a point or a puncture. So, using the construction from Section \ref{sec_pants_dec}, we can build a family of hyperbolic metrics $X_t$ on that surface such that $\ell_{X_t}(\alpha)=t$. It however follows from the Gau\ss--Bonnet theorem that the areas of the subsurfaces bounded by $\alpha$ do not depend on the choice of metric and so, letting $t$ tend to $0$, we obtain a family of metrics whose bass note tends to $0$. This fact can also be proven directly by using the variational characterization \eqref{eq_var_char} of $\lambda_1(X)$. We refer the reader to the paper by \textcite{SchoenWolpertYau} for more on the relation between short separating curves and small eigenvalues.

It's also known that, when $X$ is closed, a lower bound on $\lambda_1(X)$ and the \textbf{systole} of $X$ (the length of the shortest closed geodesic) implies an upper bound on the diameter $\mathrm{diam}(X)$ of $X$. The sharpest version of this bound, due to \textcite{Magee_letter}, states that for all $\delta\geq 0$ and $\varepsilon>0$, there exists a constant $C_{\varepsilon,\delta}>0$ such that
\[
\mathrm{diam}(X) \leq \frac{2}{1-\delta}\log(\mathrm{area}(X)) + \frac{4}{1-\delta}\log\log(\mathrm{area}(X)) + C_{\varepsilon,\delta}
\]
for all closed hyperbolic surfaces $X$ with $\lambda_1(X)\geq \frac{1-\delta^2}{4}$ and $\mathrm{systole}(X) \geq \varepsilon$. A similar bound can be derived for the diameter of the thick part of a non-compact hyperbolic surface of finite area. This is (up to a constant factor) the smallest diameter one could hope for: due to the exponential area growth, the diameter of a surface is also at least logarithmic; see the articles by \textcite{Bavard,BudzinskiCurienPetri_diameter,Mathien_diameter} for more details. Magee proved his result using an effective version of a theorem of \textcite{Ratner}, due to \textcite{Matheus} that implies that the exponential rate of mixing of the geodesic flow goes up as $\lambda_1$ goes up, which in itself is an indication that $\lambda_1$ is a measure of connectivity.

\subsection{The highest possible bass note} 

Since we now know that $\lambda_1(X)$ can be arbitrarily small in most moduli spaces, the next natural question is how large $\lambda_1(X)$ can be. In the non-compact case, our definition implies that it's at most $\frac{1}{4}$. In the closed case, this is not necessary. For example, there are very sharp numerical estimates available for the bass note of the Bolza surface and its bass note is well above $\frac{1}{4}$. Indeed,  \textcite{StrohmaierUski} proved that
\[
\lambda_1(X_{\mathrm{Bolza}}) = 3.838887 \pm 10^{-6}.
\]
It is expected that this is the closed hyperbolic surface with the highest possible bass note:
\begin{conj}
\begin{itemize}
\item[(a)] The Bolza surface uniquely maximizes the bass note among closed and orientable hyperbolic surfaces.
\item[(b)] The Klein quartic uniquely maximizes the bass note among closed and orientable hyperbolic surfaces of genus $3$.
\end{itemize}
\end{conj}
Using a method called the \textbf{conformal bootstrap}, that we shall not get into in this text, \textcite{KravchukMazacPal,Bonifacio} prove global upper bounds on the bass notes of surfaces of a fixed genus. These bounds are very close to the numerical values of the bass notes of the Bolza surface and the Klein quartic in genus $2$ and $3$ respectively, whence the conjecture.

There are multiple ways to obtain a global upper bound on $\lambda_1(X)$ for all $X\in\mathcal{M}_g$. \textcite{Huber_eigenwert} was the first to prove that
\begin{equation}\label{eq_Huber_bound}
\lim_{g\to\infty} \sup\left\{\lambda_1(X);\; X\in\mathcal{M}_g\right\} \leq \frac{1}{4}.
\end{equation}
This result was generalized by \textcite{Cheng} with a slightly different argument that also leads to the more explicit bound
\[
\lambda_1(X) \leq \frac{1}{4} + \left(\frac{4\pi}{\mathrm{diam}(X)}\right)^2 \leq \frac{1}{4}+\left(\frac{4\pi}{\log(g-1)}\right)^2
\]
Using bounds on eigenvalues of hyperbolic disks, due to \textcite{Gage}, the error term can be improved by a factor of $4$. An alternative argument, using the Selberg trace formula, due to \textcite[Theorem 8.3]{FortierBourquePetri_LPbounds}, asymptotically improves the error term by a further factor of $4$. Of course, $\frac{1}{4}$ shows up in all of these bounds because it's the bottom of the spectrum of the Laplacian on the hyperbolic plane.

The above leads to the question of whether there are sequences of closed hyperbolic surfaces whose first eigenvalue tends to the bottom of the spectrum of the hyperbolic plane. Indeed, \textcite{Buser_bipartition} conjectured that such sequences exist (correcting an earlier conjecture, also by \textcite{Buser_cubicgraphs}).

We have already mentioned in the introduction that various models of random surfaces produce sequences of closed hyperbolic surfaces whose first eigenvalue tends to $\frac{1}{4}$, thus confirming Buser's conjecture, and we will spend a large part of this text discussing these results. Random constructions are currently the only known way of building such sequences. 

Now that we know that we can approximate $\frac{1}{4}$ it's also time to become greedier and ask whether we can construct surfaces whose spectral gap is consistently at least $\frac{1}{4}$. It might be possible to do this with random surfaces by getting better control on the distribution of $\lambda_1$ around $\frac{1}{4}$, in a similar way to how this was recently done for regular graphs by \textcite{HuangMcKenzieYau} (see Sections \ref{sec_graphs} and \ref{sec_why_poly}).

\subsection{Selberg's conjecture}

Congruence surfaces are another well known source of surfaces with a uniform spectral gap. We'll briefly discuss the non-compact case $\Gamma(N)\backslash\mathbb{H}^2$. Selberg observed that these surfaces do not have non-trivial residual spectrum (see for instance \linebreak Theorem 11.3 in the book by \textcite{Iwaniec} for a proof), so understanding the spectral gap of these surfaces comes down to understanding the eigenvalues associated to cusp forms. Selberg's conjecture states that these have an optimal spectral gap.

\begin{conj}[Selberg's eigenvalue conjecture]
If we write $\lambda_1^{\mathrm{cusp}}(X)$ for the minimal eigenvalue of $\Delta_X$ on $\mathcal{C}(X)$, then we have
\[
\lambda_1^{\mathrm{cusp}}\Big(\Gamma(N) \backslash \mathbb{H}^2\Big) \geq \frac{1}{4} \quad \text{ for all }N.
\]
\end{conj}
The current best general bound, due to \textcite{KimSarnak}, is that 
\[
\lambda_1^{\mathrm{cusp}}(\Gamma(N)\backslash\mathbb{H}^2) \geq \frac{975}{4096}=0.238\ldots.
\]
The conjecture has also been verified for small $N$, by \textcite[Section 6]{Huxley} for $N \leq 18$ and by \textcite{BookerStrombergsson}
for all square free $N < 857$ and by \textcite{BookerLeeStrombergsson} for all $N\leq 226$. For the relation between Selberg's conjecture and number theoretic questions, we refer the reader to the survey by \textcite{Sarnak_Spectra}.

Huxley's proof for $N\leq 18$ is based on an idea originally due to \textcite{Roelcke}. We will give a proof using this argument for the surface we discussed in Section \ref{sec_thrice_punc}: the thrice-punctured sphere $\Gamma(2)\backslash\mathbb{H}^2$.
\begin{lemm}\label{lem_spec_thrice_punc}
We have
\[
\lambda_1^{\mathrm{cusp}}(\Gamma(2)\backslash\mathbb{H}^2) \geq \frac{2+\sqrt{2}}{4}=0.85355\ldots .
\]
\end{lemm}

\begin{proof}The geometric input we need is a convenient covering of the thrice punctured sphere by horodisks. We'll draw lifts of these horodisks in the hyperbolic plane, like in Figure \ref{pic_horodisks}.

\begin{figure}[ht]
\begin{center}
\begin{overpic}{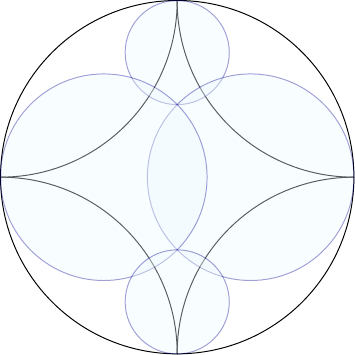}
\put(48,101.5){$1$}
\put(-4.5,48){$0$}
\put(44,-6){$-1$}
\put(101,48){$\infty$}
\end{overpic}
\caption{Four horodisks covering a fundamental domain for the thrice punctured sphere}\label{pic_horodisks}
\end{center}
\end{figure}
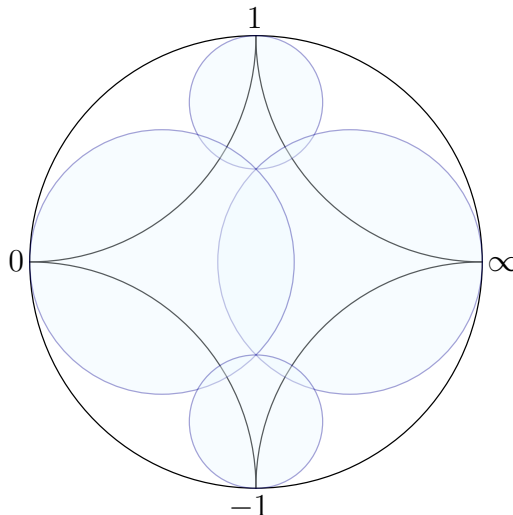

We'll take three horodisks of the same area that together cover the thrice punctured sphere and are minimal with respect to this requirement. The reason that the disks incident to $1$ and $-1$ in Figure \ref{pic_horodisks} are smaller is that these two points lie in the same parabolic orbit of $\Gamma(2)$. The horocycles are completely determined by the intersection pattern in the picture and the fact that the lengths of the two shorter horocycles are half the lengths of the longer horocycles. 

A computation yields that the horodisk incident to $\infty$ lifts to 
\[
C = \left\{ z \in \mathbb{H}^2;\; \mathrm{Im}(z) \geq 1+\frac{1}{\sqrt{2}},\; -1 \leq \mathrm{Re}(z) \leq 1 \right\}.
\] 

Now let $f:\mathbb{H}^2\to \mathbb{C}$ be a cusp form, thought of as a $\Gamma(2)$-invariant function with 
\[
\int_{\mathcal{F}} f^2\; \frac{dx\;dy}{y^2} = 1 \quad \text{and} \quad \int_{\mathcal{F}} \lVert \nabla f \rVert^2 \;  \frac{dx\;dy}{y^2} = \lambda_1^{\mathrm{cusp}}(\Gamma(2)\backslash\mathbb{H}^2),
\]
where $\mathcal{F}$ is our standard fundamental domain: the convex hull in $\mathbb{H}^2$ of $\{-1,0,1,\infty\}$. We choose some $g_1,g_2\in\mathrm{PSL}(2,\mathbb{R})$ such that $C\cup g_1 C\cup g_2 C$ projects to our horodisk covering. We observe from Figure \ref{pic_horodisks} that almost every point in $\mathcal{F}$ is covered by at most two translates of $C$, so
\begin{align*}
2 \cdot \lambda_1^{\mathrm{cusp}}(\Gamma(2)\backslash\mathbb{H}^2) & \geq  \int_{C} \lVert \nabla f\rVert^2 \frac{dx\;dy}{y^2} +   \int_{C} \lVert \nabla f\circ g_1^{-1} \rVert^2 \frac{dx\;dy}{y^2}+   \int_{C} \lVert \nabla f\circ g_2^{-1} \rVert^2 \frac{dx\;dy}{y^2}.
\end{align*}
Now let $\varphi$ be one of $f$, $f\circ g_1^{-1}$ and $f\circ g_2^{-1}$. We have $\varphi(z)=\varphi(z+2)$ for all $z\in\mathbb{H}^2$, so we can write a Fourier series
\[
\varphi(x+iy) = \sum_{n\neq 0}a_n(y) e^{inx},
\]
where we have used that $f$ is a cusp form to exclude the constant term. This means that
\begin{multline*}
\int_C \lVert \nabla \varphi \rVert^2 \; \frac{dx\;dy}{y^2}  \geq \int_C \left\lvert\frac{\partial \varphi}{\partial x}\right\rvert^2 \frac{dx\; dy}{y} = \int_C \left\lvert\sum_{n\neq 0} in\cdot a_n(y)e^{inx}\right\rvert^2 \frac{dx\; dy}{y} \\
= \sum_{n\neq 0}\int_{1+\frac{1}{\sqrt{2}}}^\infty n^2 \lvert a_n(y)\rvert^2 \frac{dy}{y} \geq (1+\frac{1}{\sqrt{2}}) \sum_{n\neq 0}\int_{1+\frac{1}{\sqrt{2}}}^\infty  \lvert a_n(y)\rvert^2 \frac{dy}{y^2}  \geq   (1+\frac{1}{\sqrt{2}}) \int_{\mathcal{F}} \lvert \varphi \rvert^2 \; \frac{dx\; dy}{y^2}.
\end{multline*}
Putting everything together yields the bound $
2\cdot \lambda_1^{\text{cusp}}(\Gamma(2)\backslash\mathbb{H}^2) \geq (1+\frac{1}{\sqrt{2}}).
$
\end{proof}

At this point, one might remember the Cheeger--Buser inequality (Theorem \ref{thm_CheegerBuser}) and hope that we could use that to prove Selberg's conjecture. This means that we would need to prove that $h(\Gamma(N)\backslash\mathbb{H}^2)\geq 1$. The Cheeger constant of the hyperbolic plane in fact satisfies $h(\mathbb{H}^2)=1$, so we'd hope to prove that the Cheeger constant of a congruence surface always exceeds that of the hyperbolic plane. \textcite{BrooksZuk} proved that this is too much to hope for and that indeed, $h(\Gamma(N)\backslash\mathbb{H}^2) \leq 0.4402\ldots $, for all large enough $N$. \textcite{BudzinskiCurienPetri_cheeger} showed that a similar thing happens for all closed surfaces, that is, $h(X) \leq \frac{2}{\pi} + o_{g\to\infty}(1)$ for all $X\in\mathcal{M}_g$. Both of these results use random decompositions of the surface.

\section{The bulk of the spectrum}\label{sec_BS_convergence}

In this text we'll be looking at surfaces with a near optimal spectral gap, i.e. surfaces whose bass note approximates the bottom of the spectrum of the hyperbolic plane. In fact, these will be sequences of surfaces whose entire Laplacian spectrum approximates the spectrum of the hyperbolic plane.

Before we get to near optimal spectral gaps, we mention a weaker notion of convergence to the hyperbolic plane, namely Benjamini--Schramm convergence. This is a geometric notion of convergence that was adapted from graph theory (see the article by \textcite{BenjaminiSchramm}) by \textcite{7samurai} and \textcite{NamaziPankkaSouto}. We say that a sequence $(X_n)_{n\geq 0}$ of hyperbolic surfaces of finite area \textbf{Benjamini--Schramm converges} to $\mathbb{H}^2$, if for all $R>0$,
\[
\lim_{n\to\infty}
\frac{
\mathrm{area}\Big\{ x\in X_n;\;\mathrm{inj}_x(X_n) \geq R\Big\} 
}
{
\mathrm{area}(X_n)
}
=1
\]
In the above, $\mathrm{inj}_x(X_n)$ denotes the injectivity radius of $X_n$ at $x$. For instance, any sequence of closed surfaces whose systole tends to infinity converges to $\mathbb{H}^2$.

The sequence $(X_n)_n$ is said to be \textbf{uniformly discrete} if the minimial injectivity radius of $X_n$ does not tend to $0$ as $n\to\infty$ (in particular, the surfaces $X_n$ eventually need to be closed).

\begin{figure}[ht]
\begin{center}
\includegraphics[scale=1]{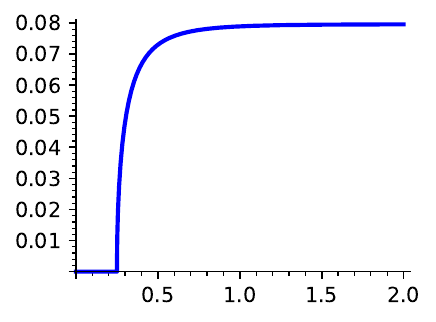}
\caption{The density of the Plancherel measure.}\label{plot_plancherel}
\end{center}
\end{figure}

It follows from the work of \textcite{7samurai} that if $(X_n)$ is a uniformly discrete sequence of hyperbolic surfaces, then for all $[a,b]\subset \mathbb{R}_{\geq 0}$:
\[
\frac{\#\{k;\lambda_k(X_n)\in[a,b]\}}{\mathrm{area}(X_n)} \quad \stackrel{n\to\infty}{\longrightarrow} \quad \frac{1}{4\pi} \int_a^b \mathds{1}_{[\frac{1}{4},\infty)}(y) \tanh\left(\pi\sqrt{y -\frac{1}{4}}\right)\; dy.
\] 

The measure $\frac{1}{4\pi} \mathds{1}_{[\frac{1}{4},\infty)}(y) \tanh\left(\pi\sqrt{y -\frac{1}{4}}\right)dy$ is the spectral density of the Laplacian on the hyperbolic plane and is called the \textbf{Plancherel measure} (see Figure \ref{plot_plancherel}). \textcite{GavelliKamp} proved that without uniform discreteness, the spectral convergence above need not hold.

Since the support of the Plancerel measure is exactly the spectrum of the Laplacian on the hyperbolic plane, we have that the number of small eigenvalues of any uniformly discrete sequence of hyperbolic surfaces $(X_n)_n$ that Benjamini--Schramm converges to the hyperbolic plane is $o(\mathrm{area}(X_n))$ as $n\to\infty$. This is of course much weaker than a near optimal spectral gap. In fact, one can construct sequences of surfaces that Benjamini--Schramm converge to $\mathbb{H}^2$ that don't have a uniform spectral gap at all, for instance using degree two covers of surfaces with large systoles (see for instance Section 5 in the article by \textcite{BuserSarnak}).
 
\section{Regular graphs}\label{sec_graphs}

Before we get to random surfaces, we discuss the case of graphs. Indeed, there is an analogous problem in graph theory to that of (optimal) spectral gaps for hyperbolic surfaces. This is the theory of (optimal) expander graphs, which is a well-studied subject in graph theory. For a general overview on expander graphs, we refer to the survey by \textcite{HooryLinialWigderson} and the book by \textcite{Kowalski}. We will discuss some recent highlights that have inspired the work on surfaces. 

There are also multiple ways to use a sequence of expander graphs to build expander surfaces (sequences of surfaces with growing area and a uniform spectral gap), for instance through covers or triangulations (see the works by \textcite{Buser_cubicgraphs,Brooks_transfer,Burger} and the notes by \textcite[Appendix]{Breuillard} for an introduction), but these estimates can't yield sharp bounds, so this is not the route we will take below. 

\subsection{Definitions}
First of all, given a locally finite graph $G$ with vertex set $V_G$ and edge set $E_G$, $A_G$ will denote its \textbf{adjacency matrix}, the matrix on $V_G\times V_G$ given by
\[
(A_G)_{vw} = \left\{ \begin{array}{ll}
\#\{\text{edges between }v \text{ and }w\} & \text{if } v\neq w \in V_G \\
2\cdot \#\{\text{loops based at }v\} & \text{if } v=w \in V_G.
\end{array}\right.
\]
In particular, our graphs are allowed to have loops and multiple edges. Since some of the graphs we will talk about have an infinite vertex set, we will think of this matrix as an operator acting on $\ell^2(V_G)$.

We will restrict to \textbf{$d$-regular graphs}: graphs such that every vertex is incident to $d$ edges (in which every loop is counted as $2$ edges). In this case the \textbf{graph Laplacian} is the operator given by
\[
L_G = d\cdot\mathrm{Id}-A_G.
\]
In many ways, this is a combinatorial analogue of the Laplacian on a Riemannian manifold. Of course, in the setting of regular graphs, any information on the $L_G$ can be extracted from $A_G$ and vice versa. As is common in graph theory, we will phrase things in terms of $A_G$ in what follows.

The adjacency matrix of a graph $G$ is a bounded and self-adjoint operator, so its spectrum is real. If $G$ is a connected $d$-regular graph on $n<\infty$ vertices, then $A_G$ is a Perron--Frobenius matrix with eigenvalues
\[
d = \lambda_1(G) > \lambda_2(G) \geq \ldots \geq \lambda_n(G).
\]
The top eigenvalue $d$ corresponds to constant vectors. $G$ is called \textbf{bipartite} if we can bipartition the vertex set $V_G = X \sqcup Y$ such that $(A_G)_{vw}=0$ whenever both $v\in X$ and $w\in X$ or both $v\in Y$ and $w\in Y$. If $G$ is finite, this happens if and only if $\lambda_1(G) = - \lambda_n(G)$, which in turn is equivalent to the entire spectrum being symmetric: $\lambda_k(G) = -\lambda_{n-k}(G)$ for $k=0,1,\ldots,n/2$.

\subsection{Expansion}

In terms of the adjacency matrix, the analogue of the bass note of a hyperbolic surface is the spectral gap $\lambda_1(G)-\lambda_2(G)=d-\lambda_2(G)$. So if we want to know how high the bass note can be, we need to know how \textbf{small} $\lambda_2(G)$ can be. Since the spectrum in this setting is finite, we could also ask for a bound on the moduli of all the eigenvalues. That is, we instead ask how small $\lambda(G) = \max\{\lambda_2(G),\lvert \lambda_n(G)\rvert \}$ can be. For example, it's the latter quantity that appears in the equidistribution rate of a random walk on a finite graph (see for instance Theorem 3.2 in the survey by \textcite{HooryLinialWigderson}). However, the former relates to the Cheeger constant of the graph in a similar way to the Cheeger--Buser inequality (Theorem \ref{thm_CheegerBuser}). A sequence of $d$-regular graphs $(G_n)_n$ is called an \textbf{expander sequence} if $\#V_{G_n} \stackrel{n\to\infty}{\to} \infty$ and there exists $\varepsilon>0$ such that $\lambda_2(G_n) \leq d-\varepsilon$ for all $n$.

In the case of $d$-regular graphs, the number $\frac{1}{4}$ is replaced by $2\sqrt{d-1}$. Indeed, the latter is the spectral radius of the adjacency matrix of the infinite $d$-regular tree -- the universal cover of every $d$-regular graph. The analogue to the Huber bound \eqref{eq_Huber_bound} is called the Alon--Boppana bound in graph theory and states
\[
\lim_{n\to\infty}\min\left\{ \lambda_2(G);\; \begin{array}{c} G \text{ a }d\text{-regular graph} \\ \text{on } n \text{ vertices} \end{array}\right\} \geq 2\sqrt{d-1}.
\]
Similar error terms to the hyperbolic case are known in this setting as well (see the articles by \textcite{Nilli,Friedman_geometric_aspects}). 

A $d$-regular graph such $G$ such that $\lambda(G) \leq 2\sqrt{d-1}$ is called a \textbf{Ramanujan graph}. Such graphs cannot be bipartite. A bipartite $d$-regular graph $G$ such that $\lambda_2(G)\leq 2\sqrt{d-1}$ is called a \textbf{bipartite Ramanujan graph}. Infinite sequences of Ramanujan graphs and bipartite Ramanujan graphs of fixed degree were first discovered by \textcite{LubotzkyPhillipsSarnak,Margulis}. These graphs are of degree $q+1$, where $q$ is a prime power (orginially, $q$ was a prime, but \textcite{Morgenstern} extended it to prime powers). It took a long time before the remaining degrees were achieved as well. \textcite{MarcusSpielmanSrivastava}, partially proving a conjecture by \textcite{BiluLinial}, iterated covers of degree $2$ to construct the first infinite sequences of bipartite Ramanujan graphs in the remaining degrees. More recently, using random regular graphs, \textcite{HuangMcKenzieYau} proved that there exist infinite sequences of Ramanujan graphs in every degree.

\subsection{Friedman's theorem and strong convergence}\label{sec_alon_conj}

This last result fits in a long line of results on the spectral gaps of random regular graphs. \textcite{Alon} conjectured that, if $G_{n,d}$ denotes a random $d$-regular graph on $n$ ($n$ even if $d$ is odd) vertices, then for every $\varepsilon>0$,
\[
\mathbb{P}\left( \lambda(G_n) \leq 2\sqrt{d-1}+\varepsilon \right) \stackrel{n\to\infty}{\longrightarrow} 1.
\]
This conjecture was first proved by \textcite{Friedman_Alon_conjecture}. By now, several proofs of Friedman's theorem exist. Indeed, also \textcite{HuangMcKenzieYau} prove a version of this theorem, with sufficient control over the distribution of $\lambda(G_n)-2\sqrt{d-1}$ in order to be able prove that $\mathbb{P}(\lambda(G_n)>2\sqrt{d-1})$ does not tend to zero as $n\to\infty$, thus proving the existence of arbitrarily large Ramanujan graphs in every degree.

We will not go into the proof by Huang, McKenzie and Yau. Instead, we will briefly discuss Friedman's original proof and more recent work by \textcite{Bordenave,BordenaveCollins} that generalizes the theorem and proves a strong convergence result.

Before we do any of this, we need to say what we mean by a random $d$-regular graph. There are multiple contiguous\footnote{Two models of finite random graphs $G_n^1$ and $G_n^2$ are called \textbf{contiguous} if and only if for any set $A$ of isomorphism classes of finite graphs, $\lim_{n\to\infty}\mathbb{P}(G_n^1\in A)=1$ if and only if $\lim_{n\to\infty}\mathbb{P}(G_n^2\in A)=1$.} models for this. We will restrict to the case of even degree and use the \textbf{permutation model} in what follows. We will write $\mathfrak{S}_n$ for the symmetric group on $[n]=\{1,\ldots,n\}$ and define the random $2d$-regular graph $G_{2d,n}$ as follows. We draw $\pi_1,\ldots,\pi_d\in\mathfrak{S}_n$ independently uniformly at random. The vertex set of $G_{2d,n}$ is $[n]$ and the number of edges between $i,j\in [n]$ equals
\[
\Big(A_{G_{2d,n}}\Big)_{ij} = \#\{k;\; \pi_k(i) =j\} + \#\{k;\; \pi_k(j) = i\}.
\] 
See Figure \ref{pic_permutation_graph} for an example. We observe that $G_{2d,n}$ is not necessarily connected. It however follows from work by \textcite{Dixon} that the probability that it is connected tends to $1$ as $n\to\infty$.

\begin{figure}[ht]
\begin{center}
\begin{overpic}{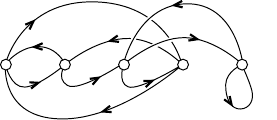}
\put(5,16){$1$}
\put(27,16){$2$}
\put(48.5,16){$3$}
\put(69.5,16){$4$}
\put(91,16){$5$}
\end{overpic}
\caption{The $4$-regular graph on $[5]$ corresponding to the permutations $(1\;2\;3\;4)(5)$ and $(1\;4\;2)(3\;5)$ from $\mathfrak{S}_5$. We've drawn arrows to indicate what the permutations do, but the graph is thought of as undirected.}\label{pic_permutation_graph}
\end{center}
\end{figure}
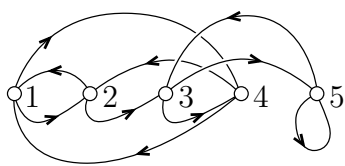
Other well known models are the uniform model (picking a uniformly random isomorphism class of $d$-regular graphs on $n$ vertices) and the configuration model (taking $n$ vertices with $d$ half-edges sticking out and randomly pairing the half-edges). When the degree is even, all three of these models are contiguous (see for instance the survey by \textcite[Section 4]{Wormald} and references therein). In particular, if Alon's conjecture holds for one of them, it holds for all of them.

\subsubsection{Friedman's proof}\label{sec_Friedman}

Friedman's proof uses the \textbf{trace method}, pioneered by \textcite{BroderShamir}. The idea is that the eigenvalues of $A_{G_{2d,n}}$ can be accessed by taking the trace of a high power:
\[
\mathrm{Tr}\left(A_{G_{2d,n}}^k\right) = d^k+\sum_{i=2}^n \lambda_i(G_{2d,n})^k.
\]
On the other hand, we can write
\[
\mathrm{Tr}\left(A_{G_{2d,n}}^k\right)  = \sum_{v=1}^n \#\Big\{\text{closed walks in }G_{2d,n} \text{ of length }k\text{ based at the vertex }v\Big\}.
\]
As such, one needs good estimates on the number of closed walks in a random graph. It follows from work by \textcite{Nica} that in the permutation model, the primitive length spectrum -- the multi-set of all lengths of primitive closed non-backtracking cycles, with multiplicity -- converges to a Poisson point process. This in particular implies that these random graphs Benjamini--Schramm converge to the $2d$-regular tree and that their normalized spectral measures converge to that of the tree, which is given by the Kesten--McKay law (see the articles by \textcite{Kesten,McKay}), the graph theoretic analogue of the Plancherel measure mentioned in Section \ref{sec_BS_convergence}. In particular, the number of eigenvalues outside of the interval $[-2\sqrt{2d-1},2\sqrt{2d-1}]$ is of size $o(n)$.

In order to prove that there are in fact none at all in the interval \linebreak $[-2\sqrt{2d-1}-\varepsilon,2\sqrt{2d-1}+\varepsilon]$ , Friedman does not directly work with closed walks but rather with non-backtracking walks, the graph theoretic analogue of closed geodesics. So we write
\[
\mathrm{Tr}\Big(A_{G_{2d,n}}^k\Big) = \sum_{l=0}^\infty f_k(l)\cdot \mathrm{Tr}\Big(P_{G_{2d,n}}^{(l)}\Big),
\]
where $P_{G_{2d,n}}^{(l)}$ is the $n\times n$ matrix in which $\Big(P_{G_{2d,n}}^{(l)}\Big)_{vw}$ equals the number of \linebreak non-backtracking walks of length $l$ from $v$ to $w$. This matrix is given by the recurrence relation
\[
P_{G_{2d,n}}^{(0)} = \mathrm{Id}, \quad P_{G_{2d,n}}^{(1)} = A_{G_{2d,n}},\quad P_{G_{2d,n}}^{(2)} =  A_{G_{2d,n}}^2 - d\cdot \mathrm{Id} 
\]
\[
\text{and } \quad P_{G_{2d,n}}^{(l+1)} = A_{G_{2d,n}} \cdot P_{G_{2d,n}}^{(l)}-(d-1)\cdot P_{G_{2d,n}}^{(l-1)} \text{ for }l\geq 1.
\]
In particular, $P^{(l)}_{G_{2d,n}} = p_l\Big(A_{G_{2d,n}}\Big)$, where $p_l$ is a polynomial of degree $l$ that can be expressed in terms of Chebycheff polynomials. The function $f_k:\mathbb{N}\to\mathbb{R}$ can be made explicit as well, we refer the reader to the article by \textcite{Friedman_Alon_conjecture} for details; we will just note that $f_k(l)\geq 0$, with equality when $l>k$, for all $l\in\mathbb{N}$. Friedman moreover proves that 
\[
f_k(l) \leq \sqrt{\frac{2d-1}{2d}} \cdot \frac{\Big(2\sqrt{2d-1}\Big)^k}{(d-1)^{l/2}} \quad \text{and} \quad f_k(0) \leq \Big(2\sqrt{2d-1}\Big)^k.
\]

We next analyze the expected traces of the non-backtracking matrices. We have
\begin{lemm}\label{lem_trace_expansion}
There exist functions $g_0,g_1,\ldots :\mathbb{N}\to \mathbb{R}$ such that for any fixed $d,r\in\mathbb{N}$ there exists a constant $c>0$ such that
\[
\left\lvert \mathbb{E}\Big(\mathrm{Tr}\Big(P^{(l)}_{G_{2d,n}}\Big)\Big) - g_0(l) - \frac{g_1(l)}{n} - \ldots - \frac{g_{r-1}(l)}{n^{r-1}} \right\rvert \leq c\cdot \frac{l^{4r+2}\cdot (2d-1)^l}{n^r} \quad \text{for all }n\geq 1.
\]
\end{lemm}

The next step in the strategy is to analyze the functions $g_j$. Indeed, it would suffice to prove that these functions are what Friedman calls \textbf{Ramanujan functions} with principal term $2d(2d-1)^{l-1}$ for $g_0$ and $0$ for all the others. Here, a Ramanujan function is a function $g:\mathbb{N}\to\mathbb{R}$ such that there exists a polynomial $p:\mathbb{N}\to\mathbb{R}$ and a constant $c>0$ such that
\[
\lvert g(l) - (2d-1)^l \cdot p(l) \rvert \leq c\cdot l^c \cdot (2d-1)^{l/2}.
\]
The function $l\mapsto (2d-1)^l \cdot p(l)$ is called the \textbf{principal term} of $g$. We now have:

\begin{lemm}
Let $g_0,g_1,\ldots, g_{r-1}$ be as in Lemma \ref{lem_trace_expansion}. If these are Ramanujan functions and the principal term of $g_m$ equals $2d(2d-1)^{l-1}$ when $m=0$ and $0$ otherwise, then,
\[
\mathbb{E}\Big( \sum_{i=2}^n\lambda_i(G_{2d,n})^k \Big) \leq \Big(2\sqrt{2d-1}\Big)^k \cdot \left(\frac{d}{2\sqrt{2d-1}}\right)^{\frac{k}{r+1}}\cdot \left(1+ \frac{c \cdot \log\log(n)}{\log(n)}\right)^k
\] 
for all $k\leq \frac{(r+1)\log(n)}{\log(2d/(2\sqrt{2d-1})}$.
\end{lemm}

\begin{proof}[Proof sketch]
Being a polynomial in $A$, $P_{G_{2d,n}}^{(l)}$ has the same eigenspaces as $A$. As such, writing $\mathds{1}$ for the vector all of whose entries equal $1$,
\begin{multline*}
\mathbb{E}\Big( \sum_{i=2}^n\lambda_i(G_{2d,n})^k \Big) = \mathbb{E}\left( \mathrm{Tr}\left(\Big(A_{G_{2d,n}}\rvert_{\mathds{1}^{\perp}} \Big)^k \right)\right) \\
 = \sum_{l=0}^k f_k(l)\cdot \mathbb{E} \left( \mathrm{Tr}\left(P_{G_{2d,n}}^{(l)}\rvert_{\mathds{1}^{\perp}} \right)\right) = \sum_{l=0}^k f_k(l)\cdot \left(\mathbb{E} \left( \mathrm{Tr}\left(P_{G_{2d,n}}^{(l)}\right) \right) - p_l(2d)\right)
\end{multline*}
now we apply Lemma \ref{lem_trace_expansion} and observe that terms of the form $p_l(2d)=2d(2d-1)^{l-1}$ cancel against the principal term of $g_0(l)$. Combining the remainder with the estimate we have on $f_k(l)$ then implies the claim.
\end{proof}

If we could prove the hypothesis of the lemma above for all $r$, we would get Alon's conjecture by applying Markov's inequality to $\lambda(G_{2d,n})^k-(2\sqrt{2d-1}-\varepsilon)^k$, which, when $n$ is large enough, is a non-negative random variable because of the Alon--Boppana theorem. However, as Friedman proves, there exists a universal constant $C>0$ such that when $r>C\sqrt{d}\log(d)$, the hypothesis is violated. That is, the functions $g_i$ are not Ramanujan in this range. 

What prevents them from being Ramanujan is the presence of \textbf{tangles}; these are small connected subgraphs of negative Euler characteristic that cause the spectral gap to be small. These tangles appear with a relatively small probability (that in particular tends to $0$ as $n$ tends to infinity). However, these probabilities are large enough to ruin the expectations we are calculating. A large part of the work in Friedman's article goes into dealing with these tangles. In particular, the goal is to take the expectation over a subset of graphs that are still of asymptotic density $1$ but do not contain these problematic subgraphs. We will not go into details here, but just note that a similar problem appears in the work by Anantharaman and Monk (see Section \ref{sec_AM}).

\subsubsection{Representations, covers and strong convergence}\label{sec_bordenavecollins}

In this section, we will describe the generalization by \textcite{BordenaveCollins} of Friedman's theorem (which was in part proved using the strategy \textcite{Bordenave} developed for his new proof of Friedman's theorem). We will not sketch this proof, but we will describe the theorem statement, because it's used in the work by Hide--Magee that we present in Section \ref{sec_HM}. A new approach,  that we won't discuss in detail but plays an important role in the work by Hide--Macera--Thomas (see Section \ref{sec_HMT}), to the result by Bordenave--Collins has recently been found by \textcite{ChenGarzaVargasTroppvanHandel}.

We start with a group theoretic interpretation of the permutation model. We let $F_d$ denote a non-abelian free group of rank $d$. If $(x_1,\ldots,x_d)$ is a generating set for $F_d$  (which because of its cardinality is necessarily a free basis), then the map
\[
\varphi\in\mathrm{Hom}(F_d,\mathfrak{S}_n) \quad \longmapsto \quad (\varphi(x_1),\ldots,\varphi(x_d)) \in \Big(\mathfrak{S}_n\Big)^d
\]
is a bijection. In particular the choice of $d$ random permutations we used for our random graph also yields a random homomorphism $\varphi_n:F_d\to \mathfrak{S}_n$. Moreover, our graph $G_{2d,n}$ is nothing else than the \textbf{Schreier graph} for the action of $F_d$ on $[n]$ through $\varphi_n$. 

We can also understand the adjacency matrix in this language. If $\mathrm{U}(n)$ is the unitary group of $\mathbb{C}^n$ with its standard hermitian product and $\mathrm{perm}_n:\mathfrak{S}_n \to \mathrm{U}(n)$ the permutation action of $\mathfrak{S}_n$ on $\mathbb{C}^n$, then
\[
A_{G_{2d,n}} = \sum_{i=1}^d \mathrm{perm}\circ\varphi_n(x_i)+\mathrm{perm}\circ\varphi_n(x_i^{-1}).
\]
The representation $\mathrm{perm}_n$ splits into two irreducible representations
\[
\mathrm{perm}_n \simeq \mathrm{triv} \oplus \mathrm{std}_n,
\] 
where $\mathrm{triv}$ denotes the trivial representation and is obtained as the restriction of $\mathrm{perm}_n$ to $\mathbb{C}\cdot \mathds{1}\subset \mathbb{C}^n$, where $\mathds{1}$ denotes the vector with only $1$'s, and $\mathrm{std}_n:\mathfrak{S}_n \to \mathrm{U}(\mathbb{C}^n_0)$ is the restriction of $\mathrm{perm}_n$ to
\[
\mathbb{C}^n_0 = \left\{ v\in\mathbb{C}^n;\; \sum_{i=1}^n v_i =0 \right\}
\]
and is often called the \textbf{standard representation}.

This decomposition into irreducible representations also corresponds exactly to the decomposition of $\mathbb{C}^n$ into the trivial eigenspace of $A_{G_{2d,n}}$ and its orthogonal complement. Of course, this is not an accident. Indeed, the graph $G_{2d,n}$ can be thought of as a \linebreak (non-regular) cover of the wedge of $d$ circles $G_{2d,1}$. If we write $G_{2d,1} = F_d \backslash \mathcal{T}_{2d}$, where $\mathcal{T}_{2d}$ denotes the infinite $2d$-regular tree, then
\[
G_{2d,n} \simeq F_d\backslash\Big(\mathcal{T}_{2d}\times [n]\Big),
\]
where $F_d$ acts on the second factor through the map $\varphi_n:F_d\to\mathfrak{S}_n$. In particular, the homomorphism $\varphi_n$ is the monodromy representation of the cover $G_{2d,n}\to G_{2d,1}$. Thinking of vectors in $\mathbb{C}^n$ (and $\mathbb{C}^1$) as functions on the vertex set $V_{G_{2d,n}}$ (and $V_{G_{2d,1}}$ respectively), the eigenfunctions of the $1\times 1$ matrix $A_{G_{2d,1}} = (2d)$ all lift to eigenfunctions of $A_{G_{2d,1}}$. These lifts form the $1$-dimensional subspace of constant functions. The orthogonal complement are thus exactly the ``new'' functions on $G_{2d,n}$. We note this principle has nothing to do with graphs and that a similar decomposition of function spaces works for covers of topological spaces in general. We will see another instance of this when discussing the work of Hide--Magee (Section \ref{sec_HM}).

The above discussion implies that if we want to prove a spectral gap, we need to control the spectral radius
\[
\lVert A_{G_{2d,n}}^{\mathrm{new}} \rVert_\infty := \left\lVert \mathrm{std}_n \circ \varphi_n \left( \sum_{i=1}^d x_i+x_i^{-1} \right) \right\rVert_\infty
\]
of the self-adjoint matrix $A_{G_{2d,n}}^{\mathrm{new}}$, where we have linearly extended $\mathrm{std}_n\circ\varphi_n$ to the group ring $\mathbb{C}[F_d]$ -- the ring of formal finite linear combinations of elements of $F_d$. We need the following definition, in which $\mathrm{U}(\mathcal{H})$ denotes the group of unitary operators on a Hilbert space $\mathcal{H}$:
\begin{defi}
Let $\Gamma$ be a countable group and let $\Big(\rho_n:\Gamma \to \mathrm{U}(\mathcal{H}_n)\Big)_{n\geq 1}$ be a sequence of representations on a sequence of Hilbert spaces $\mathcal{H}_n$. We say that $\rho_n$ \textbf{converges strongly} to a representation $\rho:\Gamma\to U(\mathcal{H})$ on a Hilbert space $\mathcal{H}$ if, for all $x\in\mathbb{C}[\Gamma]$,
\[
\lim_{n\to\infty} \lVert \rho_n(x) \rVert_\infty = \lVert \rho(x) \rVert_\infty.
\]
If the representations $\rho_n$ are random we say the sequence \textbf{converges strongly in probability} to $\rho:\Gamma\to \mathrm{U}(\mathcal{H})$ if, for all $x\in\mathbb{C}[\Gamma]$ and all $\varepsilon>0$,
\[
\lim_{n\to\infty} \mathbb{P}\Big(\;\big\lvert\;\lVert \rho_n(x) \rVert_\infty - \lVert \rho(x) \rVert_\infty \;\big\rvert < \varepsilon\Big) = 1.
\]
\end{defi}

The case that is of interest to us is when the limit is the \textbf{left regular representation} of $\Gamma$. This is the representation $\rho_{\mathrm{reg}.}:\Gamma \to \mathrm{U}(\ell^2(\Gamma))$ given by
\[
\Big(\rho_{\mathrm{reg}.}(\gamma)\cdot f\Big)(\eta) = f(\gamma^{-1}\cdot\eta),\quad \gamma,\eta\in \Gamma, f\in\ell^2(\Gamma).
\]

\textcite{BordenaveCollins} proved the following:

\begin{theo}[Bordenave--Collins]\label{thm_BordenaveCollins}
Let $d\geq 2$ be fixed and let $\varphi_n\in\mathrm{Hom}(F_d,\mathfrak{S}_n)$ be uniformly random. Then $\mathrm{std}_n\circ\varphi_n$ converges strongly in probability to the left regular representation $\rho_{\mathrm{reg.}}:F_d \to \mathrm{U}(\ell^2(F_d))$ as $n\to\infty$.
\end{theo}

We finish this section by explaining how the theorem above implies Friedman's theorem. First we recall what a \textbf{Cayley graph} is. Given a group $\Gamma$ and a generating set $S$ for $\Gamma$, the Cayley graph $G(\Gamma,S)$ is the graph with vertex set $V_{G(\Gamma,S)}=\Gamma$ and edge set
\[
E_{G(\Gamma,S)} = \left\{ \{\gamma,\gamma \cdot x);\; x \text{ or }x^{-1} \in S\right\}.
\]
The Cayley graph of $F_d$ with respect to the generating set $(x_1,\ldots,x_d)$ is isomorphic to $\mathcal{T}_{2d}$. In particular, 
\[
\rho_{\mathrm{reg.}}\left(\sum_{i=1}^d x_i + x_i^{-1}\right) = A_{\mathcal{T}_{2d}}.
\]
So if we apply Theorem \ref{thm_BordenaveCollins} to $\sum_{i=1}^d x_i + x_i^{-1} \in \mathbb{C}[F_d]$, we obtain Friedman's theorem. We however emphasize that the theorem above is much stronger than what we need for a near optimal spectral gap for random graphs. When we apply it to random surfaces below, we will need the full strength of the statement.

Bordenave and Collins's result is also part of a much longer story on strong convergence for unitary representations. Strikingly, even though we now know, through Theorem \ref{thm_BordenaveCollins}, that the regular representation of $F_d$ can be strongly approximated by finite dimensional unitary representations that factor through standard representations of symmetric groups, we currently don't have any explicit examples of this. In fact, even if we don't ask for the representations to factor through the symmetric group, no explicit sequences are known. We refer to the surveys by \textcite{Magee_survey,vanHandel_survey} for more on this.

\section{Random surfaces} 

Random surfaces have been around in the mathematical physics, combinatorics and probability literature for a long time. One particular model that has attracted a lot of attention is that of \textbf{random planar maps} (random cell decompositions of the $2$-sphere). Some of the highlights in this field are the description of the local geometry, by \textcite{AngelSchramm}, and the global geometry, by \textcite{LeGall,Miermont}, of random planar maps with faces with a fixed number of sides.

In this section we will present a brief survey of the history of random hyperbolic surfaces. We will introduce the three most common models in the order in which they have historically been introduced, which is the opposite order to that in which the near optimal spectral gaps for them were proved. Towards the end of this section, we discuss some alternative models.

\subsection{Random Bely\u{\i} surfaces}\label{sec_BM_model}

\subsubsection{The model} The first model of random hyperbolic surfaces was defined by \textcite{BrooksMakover}. In this model, the first step is to randomly glue together $2n$ ideal hyperbolic triangles without shear (this can be defined by requiring that the tangency points of the maximal inscribed disks match up, like in Section \ref{sec_thrice_punc}) into a surface $S_n$ with cusps. Figure \ref{pic_brooks_makover} shows a sample for $n=3$. The gluings themselves are determined by the shear condition, so the only randomness comes from the combinatorics of the triangulation. The model Brooks and Makover use for this is often called the \textbf{configuration model}: the sides of the triangles are labeled with the numbers $1,\ldots,6n$ and the combinatorics of the gluing is determined by a uniformly random perfect matching (a set partition into pairs) on $[6n]$.

\begin{figure}[ht]
\begin{center}
\includegraphics{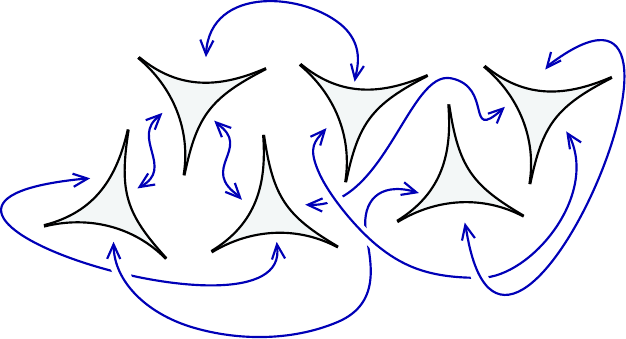}
\caption{Gluing ideals triangles into a surface without boundary.}\label{pic_brooks_makover}
\end{center}
\end{figure}

Afterwards, we compactify $S_n$ so as to obtain a surface $\overline{S_n}$. To do this, we observe that the cusps admit neighborhoods that are conformally equivalent to punctured disks in the complex plane and one can simply add the punctures in. This yields a closed Riemann surface and hence, using the uniformization theorem, a constant curvature metric.

As observed by \textcite{BrooksMakover_Belyi}, this model can also be described as a random cover model. To see this, we recall that as abstract groups 
\[\mathrm{PSL}(2,\mathbb{Z}) \simeq \left(\mathbb{Z}/2\mathbb{Z}\right) \ast \left(\mathbb{Z}/3\mathbb{Z}\right)=\langle \tau,\sigma|\; \tau^2=\sigma^3=1\rangle\]
where ``$\ast$'' denotes a free product. As such, we can build a random cover of the modular curve $\mathrm{PSL}(2,\mathbb{Z})\backslash\mathbb{H}^2$ using a random homomorphism to a symmetric group, in the same way we did for graphs. In order to have these covers be surfaces (as opposed to orbifolds), we build a random homomorphism $\varphi_n:\mathrm{PSL}(2,\mathbb{Z})\to \mathfrak{S}_{6n}$ by sending $\tau$ to a uniformly random fixed point free involution (so a product of $3n$ distinct $2$-cycles) and $\sigma$ to a uniformly random permutation of order three without fixed point (so a product of $2n$ distinct $3$-cycles). This guarantees that the torsion doesn't lift to the surface
\[
S_n = \mathrm{PSL}(2,\mathbb{Z}) \backslash \Big(\mathbb{H}^2\times [6n]\Big),
\]
where $\mathrm{PSL}(2,\mathbb{Z})$ acts on $[6n]$ through $\varphi_n$. We're using the same notation because distribution of this surface $S_n$ is the same as that of the cusped surface $S_n$ we defined before.

When we compactify $S_n$ and $\mathrm{PSL}(2,\mathbb{Z})\backslash\mathbb{H}^2$, the cover $S_n\to\mathrm{PSL}(2,\mathbb{Z})\backslash\mathbb{H}^2$ turns into a branched cover, branched over at most three points (coming from the two cone points and the cusp). By \textcite{Belyi}, the set of surfaces that admit such a cover are exactly the curves that can be defined over $\overline{\mathbb{Q}}$, a dense set among all Riemann surfaces (see the survey by \textcite{JonesSingerman} for more on this). This means that the most interesting surfaces are the compactifications $\overline{S_n}$. Indeed, the set of all non-compact surfaces that the model generates intersect every moduli space only in a finite number of points. 

In order to understand the geometry of the surfaces $\overline{S_n}$, one needs to be able to compare the geometry of the non-compact surfaces $S_n$, whose geometry is entirely encoded by the combinatorics of the triangulation, and the geometry of their compactifications $\overline{S_n}$. \textcite{Brooks_platonic} proved effective comparison theorems for the geometry away from the cusps and their images before and after compactification that can be used to do this.

\subsubsection{Results}
Now we first need to understand the topology of $\overline{S_n}$. It follows from classical results on the configuration model by \textcite{Bollobas,Wormald_connected} that the surface $\overline{S_n}$ is connected with high probability as $n\to\infty$\footnote{We will say that a sequence of events $(A_n)_n$ hold with high probability as $n\to\infty$ if $\mathbb{P}(A_n)\to 1$ as $n\to\infty$.}. \textcite{BrooksMakover} proved that with high probability the genus of $S_n$ is close to its maximal value $\frac{n+1}{2}$. \textcite{Gamburd} refined this result and obtained very precise estimates on the distribution of the genus. He did so by applying a famous lemma by \textcite{DiaconisShahshahani} to the distribution of $\varphi_n(\sigma\tau)$, which describes how the cusp of $\mathrm{PSL}(2,\mathbb{Z})\backslash\mathbb{H}^2$ lifts to $S_n$.  All this in particular implies that $\overline{S_n}$ is typically hyperbolic.

The first results on the geometry are by \textcite{BrooksMakover}, who proved that their random surfaces are expanders. They proved that there exist constants $C_1,C_2,C_3,C_4>0$ such that
\[
\mathrm{diam}(\overline{S_n}) \leq C_1\cdot \log(n), \quad \lambda_1(\overline{S_n}) \geq C_2, \quad h(\overline{S_n}) > C_3 \quad \text{and} \quad \mathrm{systole}(\overline{S_n})>C_4, 
\]
with high probability as $n\to\infty$. All of these bounds have since been sharpened. \textcite{BudzinskiCurienPetri_diameter_Belyi} proved that the diameter is asymptotic to $2\log(n)$ in probability as $n\to\infty$, \textcite{Petri,PetriThaele} proved that the primitive length spectrum -- the multi-set of lengths of primitive closed geodesics on $\overline{S_n}$ -- converges to a Poisson point process. This also implies uniform discreteness and Benjamini--Schramm convergence of these surfaces, see the survey by \textcite{Raimbault}. \textcite{ShenWu1,ShenWu2} proved that the Cheeger constant is at most $\frac{3}{2\pi}+o(1)$ and the spectral gap tends to $\frac{1}{4}$, as $n\to\infty$, using methods close to those of \textcite{HideMagee} for the latter. Finally, \textcite{Naud} proved a general criterion for the convergence of the normalized $\zeta$-regularized determinant of the Laplacian to a universal value, that in particular applies to this model of random surfaces.

\subsection{Weil--Petersson random surfaces}\label{sec_WP}

\subsubsection{The model} The second model we will treat is that of surfaces sampled using the Weil--Petersson volume form on $\mathcal{M}_g$ for $g\geq 2$. \textcite{Wolpert} proved that, for any pants decomposition $\mathcal{P}$ of a closed surface of genus $g$, the symplectic form 
\[
\sum_{\alpha\in\mathcal{P}} d\ell_\alpha \wedge d \tau_\alpha
\]
coincides with the symplectic form coming from the \textbf{Weil--Petersson} Kähler form on Teichmüller space, the definition of which we will not discuss in this text. Wolpert's theorem implies that the volume form $\wedge_{\alpha \in\mathcal{P}} (d\ell_\alpha\wedge d\tau_\alpha)$ defines a mapping class group invariant volume form on Teichmüller space and thus a volume on any moduli space of closed surfaces. We will denote the associated volume measure by $\mathrm{vol}_{\mathrm{WP}}$. The total Weil--Petersson volume of moduli space is finite, which allows us to define a probability measure, by normalizing the Weil--Petersson volume measure. We will write $X_g\in\mathcal{M}_g$ for a random surface distributed according to this measure. Similar volume forms can be defined on moduli spaces of hyperbolic surfaces with cusps and/or boundary. These forms coincide with the symplectic form defined by \textcite{Goldman}. For more background on the Weil--Petersson metric, we refer the reader to the books by \textcite{ImayoshiTaniguchi} and \textcite{Wolpert_notes}.

What sets this model apart among the three models of random surfaces we describe in this text is that the support of the corresponding probability measure is \textbf{all} hyperbolic surfaces of a given topology. We note that there are other natural such models that are currently poorly understood, like for instance the model coming from the Weil--Petersson measure multiplied by the Mirzakhani function (see for instance the paper by \textcite{AranaHerreraAthreya}).

\subsubsection{Results}
The geometry of Weil--Petersson random surfaces was first investigated by \textcite{GuthParlierYoung,Mirzakhani_random}. Guth, Parlier and Young used them to prove an existence result, namely that, for every $\varepsilon>0$ and every $g$ large enough, there exist hyperbolic surfaces of genus $g$ that do not admit a pants decomposition of total length $\leq g^{7/6-\varepsilon}$. This is still the best known lower bound on the maximal minimal total pants length of a surface of genus $g$. The best known upper bound follows from work by  \textcite{Parlier_shorter_note} on the constant in \textcite{Bers}. The resulting bound states that every closed hyperbolic surface admits a pants decomposition of total length at most $12\pi(g-1)^2$.

Mirzkhani used the integration techniques she developed as a part of her thesis (see the article by \textcite{Mirzakhani_WP_volumes} for the beginning of this story) to study the random surfaces $X_g$. A priori, one might want to simply integrate functions over some fundamental domain $\mathcal{F}_g\subset\mathcal{T}_g$ for the corresponding mapping class groups. However, no sequence of fundamental domains is known that is also easy to describe.

Mirzakhani overcame this difficulty by proving the \textbf{Mirzakhani integration formula}. In order to state it, suppose that $\Gamma = (\gamma_1,\ldots,\gamma_k)$ is a collection of disjoint, pairwise non-homotopic, essential (see footnote \ref{fn_essential_curve}) simple closed curves on $\Sigma_{g,n}(L)$ and that $F:[0,\infty)^k \to \mathbb{R}$ is a function. Then we define $F^\Gamma:\mathcal{M}_{g,n}(L) \to \mathbb{R}$ by
\[
F^\Gamma(X) = \sum_{(\alpha_1,\ldots,\alpha_k) \in \mathrm{MCG}(\Sigma_{g,n}(L))\cdot \Gamma} F(\ell_{\alpha_1}(X),\ldots,\ell_{\alpha_k}(X)),\quad X \in \mathcal{M}_{g,n}(L).
\]
For example, if $\gamma\subset \Sigma_g$ is a non-separating simple closed curve and $F=\mathds{1}_{[0,L]}$ is the indicator function of the interval $[0,L]$, then $F^\gamma:\mathcal{M}_g\to \mathbb{N}$ is the function that assigns the number of non-separating simple closed geodesics of length at most $L$ on $X$ to $X\in\mathcal{M}_g$. 

Using a covering argument, Mirzakhani proved an integration formula for functions of the form above. In this theorem, $\mathcal{M}_{\Sigma_{g,n}(L),\Gamma}(L_1,\ldots,L_n,x_1,x_1,\ldots, x_k,x_k)$ denotes the moduli space of hyperbolic metrics on $\Sigma_{g,n}(L)-\Gamma$, in which the two boundary components coming from $\gamma_i$ both have length $x_i$ for $i=1,\ldots ,k$. If $\Sigma_{g,n}(L)-\Gamma$ has multiple connected components, then $\mathcal{M}_{\Sigma_{g,n}(L),\Gamma}(L,x,x)$ is a product over these components. Mirzakhani proved:

\begin{theo}[Mirzakhani Integration Formula]\label{thm_Mirzakhani_integration_formula}
Let $F:[0,\infty)^k\to\mathbb{R}$, let $L\in [0,\infty)^n$ and let $\Gamma=(\gamma_1,\ldots,\gamma_k)$ be a collection of disjoint pairwise non-homotopic, essential simple closed curves on $\Sigma_{g,n}(L)$. Then
\begin{multline*}
\int_{\mathcal{M}_{g,n}(L)} F^\Gamma(X)\; d\mathrm{vol}_{\mathrm{WP}}(X) \\
=C_\Gamma \int\limits_{[0,\infty)^k} F(x) \cdot \mathrm{vol}_{\mathrm{WP}}\Big(\mathcal{M}_{\Sigma_{g,n}(L),\Gamma}(L_1,\ldots,L_n,x_1,x_1,\ldots, x_k,x_k) \Big) \; x_1\cdots x_k\; dx_1 \cdots dx_k.
\end{multline*}
Here $C_\Gamma$ is a constant that depends on $\Gamma$ alone\footnote{For a precise description of this contant, we refer to the survey by \textcite{Wright_tour}.}.
\end{theo}

This theorem can be thought of as a hyperbolic version of the integration formula for the moduli space of lattices $\mathrm{SL}(n,\mathbb{Z})\backslash\mathrm{SL}(n,\mathbb{R})$ proved by \textcite{Siegel}. This is one of multiple surprising similarities to the moduli space of lattices, Indeed, the moduli space of hyperbolic surfaces is not a locally homogeneous space so there is no reason to expect it should behave similarly.

\textcite{Mirzakhani_WP_volumes,Mirzakhani_scg,Mirzakhani_intersection} combined this formula with her generalization of McShane's identity and multiple other beautiful ideas to derive recurrences for Weil--Petersson volumes, prove asymptotic counting results for the number of simple closed geodesics on a hyperbolic surface and re-prove Witten's conjecture on intersection numbers of tautological classes on the moduli space of curves. \textcite{MirzakhaniZograf} used Mirzakhani's recurrences to derive an asymptotic equivalent for $\mathrm{vol}_{\mathrm{WP}}(\mathcal{M}_{g,n})$ as $g\to\infty$ that plays an important role in nearly all the papers on Weil--Petersson random surfaces.

For the random surfaces $X_g$, \textcite{Mirzakhani_random} proved that with high probability as $g\to\infty$, for every $\varepsilon>0$,
\[
\mathrm{diam}(X_g) \leq 40\cdot\log(g),\quad h(X_g) \geq \frac{\log(2)}{2\pi + \log(2)}-\varepsilon
\]
and as a result of the latter (using the Cheeger inequality, Theorem \ref{thm_CheegerBuser}),
\[
\lambda_1(X_g) \geq \frac{\log^2(2)}{(4\pi+\log(4))^2}-\varepsilon=0.00246\ldots-\varepsilon.
\]
Mirzakhani's results also imply the Benjamini--Schramm convergence to $\mathbb{H}^2$ of these surfaces as $g\to\infty$ (see the survey by \textcite{Raimbault}). 

There has been a lot of work on these random surfaces since. We start with lengths of geodesics. \textcite{MirzakhaniPetri} proved that, like for random Bely\u{\i} surfaces, the primitive length spectrum converges to a Poisson point process as the genus tends to infinity, but with a different density (that also appears in a combinatorial setting, see the works of \textcite{JansonLouf,BarazerGiacchettoLiu}). Lengths of longer geodesics on these random surfaces were investigated by \textcite{NieWuXue,ParlierWuXue,MonkThomas,WuXue2,HeShenWuXue}.

The study of the bass note of these surfaces has also attracted a lot of attention. Mirzakhani's bound on the first Laplacian eigenvalue was first improved to $\frac{3}{16}-\varepsilon$, independently by \textcite{WuXue1,LipnowskiWright}, then to $\frac{2}{9}-\varepsilon$ by \textcite{AnantharamanMonk5}, afterwards to $\frac{1}{4}-\varepsilon$, also by \textcite{AnantharamanMonk4} (see Section \ref{sec_AM}) and most recently \textcite{HideMaceraThomas} proved that, for some fixed $c>0$, $ \lambda_1(X_g) \geq \frac{1}{4} - g^{-c}$ with high probability as $g\to\infty$ (see Section \ref{sec_HMT}). 

There are other spectral questions to ask as well. \textcite{Monk} made the Benjamini--Schramm convergence and the resulting spectral convergence of the surfaces $X_g$ more effective and \textcite{Gong,MonkRares} extended this to Laplacians twisted by a harmonic $1$-form and Dirac operators respectively. \textcite{LeMassonSahlsten} proved a quantum ergodicity theorem for these random surfaces and \textcite{GilmoreLeMassonSahlstenThomas,Thomas} studied delocalization for eigenfunctions. The latter group for instance proved that the $L^\infty$-norm of any $L^2$-normalized eigenfunction corresponding to an eigenvalue in a fixed window $[a,b]\subset (0,\infty)$ is at most $O(\log(g)^{-1/2})$.

The criterion for the convergence of the determinant of the Laplacian by \textcite{Naud} applies here too and \textcite{HeWu} also studied the expected determinant. \textcite{Rudnick,RudnickWigman1,MarklofMonk,RudnickWigman2} showed that, when first sending the genus and then the energy window to infinity, certain statistics for eigenvalues in the given window behave like those of large GOE and GUE matrices. 

Finally, we can add cusps to the picture. \textcite{Hide1,HeWuXue} studied the effect on the first eigenvalue of adding a limited number of cusps to the random surfaces. If instead of the genus, we let the number of cusps tend to infinity, the geometry and spectra of these random surfaces starts behaving very differently, see the works by \textcite{HideThomas2,HideThomas1,ShenWu3,BuddCurien}. These works prove that there are many small eigenvalues and the surfaces admit a scaling limit, which is a Brownian map, just like for random planar maps. We can also let the number of cusps and the genus tend to infinity at the same time. \textcite{BuddLions}, using work by \textcite{BuddZonneveld}, showed that when the number of cusps tends to infinity sufficiently quickly and only certain geodesics -- called tight geodesics -- are counted, their statistics on surfaces of large genus with many cusps are the same as on closed surfaces of large genus.

\subsection{Random finite degree covers}\label{sec_rand_covers}

\subsubsection{The model}
The third model we will discuss is that of random covers. This is the analogue of the permutation model for random regular graphs. In general, a random cover of a nice enough space $X$ with a finitely generated fundamental group $\Gamma = \pi_1(X,x)$ can be built by choosing a homomorphism $\varphi_n\in\mathrm{Hom}(\Gamma,\mathfrak{S}_n)$ uniformly at random and setting
\[
X_n = \Gamma \; \backslash \; (\widetilde{X} \times [n] ),
\]
where $\widetilde{X}$ denotes the universal cover of $X$ and where the action on $[n]$ is through $\varphi_n$. Indeed, because $\Gamma$ is finitely generated, $\mathrm{Hom}(\Gamma,\mathfrak{S}_n)$ is finite and thus we have a well-defined uniform probability measure on it. In general, saying anything about these random covers is very difficult. Even just estimating $\#\mathrm{Hom}(\Gamma,\mathfrak{S}_n)$ is a notoriously hard problem. We refer to the book by \textcite{LubotzkySegal} for more on this problem in general.

In our setting, the base will be an orientable hyperbolic surface of finite area. In particular, the fundamental group $\Gamma$ will either be a non-abelian free group or a surface group. In the former case, the probability space is exactly the same as in the permutation model (see Section \ref{sec_alon_conj}). That is, as sets we can make the identification $\mathrm{Hom}(F_r,\mathfrak{S}_n) = \mathfrak{S}_n^r$. Figure \ref{pic_permutation_surface} shows an example of a cover of the thrice punctured sphere (so $\Gamma\simeq F_2$) defined by a homomorphism $\varphi_5:\Gamma\to\mathfrak{S}_5$.

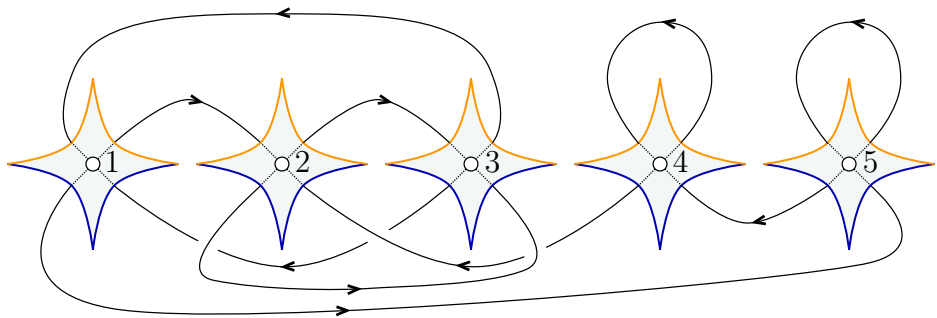
\begin{figure}
\begin{center}
\begin{overpic}{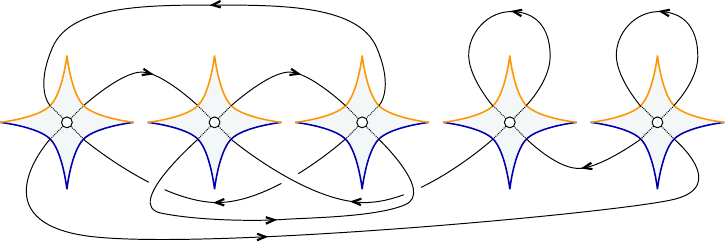}
\put(10.5,15.5){$1$}
\put(31,15.5){$2$}
\put(51.5,15.5){$3$}
\put(71.75,15.5){$4$}
\put(92.25,15.5){$5$}
\end{overpic}
\caption{Gluing ideal squares together according to the permutations $(1\;2\;3)(4)(5)$ and $(1\;5\;4\;2\;3)$ in $\mathfrak{S}_5$ yields a $5$-fold cover of the thrice punctured sphere.}\label{pic_permutation_surface}
\end{center}
\end{figure}

The closed case is significantly more complicated. If we write $\Gamma_g$ for the fundamental group of an orientable surface of genus $g$, then we have
\[
\mathrm{Hom}(\Gamma_g,\mathfrak{S}_n) = \Big\{(\alpha_1,\ldots,\alpha_g,\beta_1,\ldots,\beta_g)\in\mathfrak{S}_n^{2g} ;\; \prod_{i=1}^g [\alpha_i,\beta_i] = e \Big\}. 
\]
\textcite{Hurwitz} was the first to prove that this can be expressed in terms of the irreducible representations of the symmetric group:
\[
\#\mathrm{Hom}(\Gamma_g,\mathfrak{S}_n) = (n!)^{2g-1} \sum_{\substack{\rho:\mathfrak{S}_n\to\mathrm{GL}(V^\rho) \text{ irreducible} \\ \text{representation }}} \frac{1}{\dim(V^\rho)^{2g-2}}.
\]
\textcite{Lulov,MuellerPuchta1} analyzed the sum over the irreducible representations and proved that $\#\mathrm{Hom}(\Gamma_g,\mathfrak{S}_n)\sim 2\cdot (n!)^{2g-1}$ as $n\to\infty$ and \textcite{LiebeckShalev,MuellerPuchta2} proved a similar asymptotic equivalent for more general Fuchsian groups. Note that the $2$ in this estimate comes from the trivial representation and the sign representation of $\mathfrak{S}_n$. The underlying theorem is that asymptotically, the contribution of the higher-dimensional representations is negligible.

\subsubsection{Results}
The work by \textcite{Dixon,Nica} implies that if the base surface $X$ is not compact (but of finite area), then a random degree $n$ cover $X_n$ of $X$ is connected with high probability and its primitive length spectrum converges to a Poisson point process. The distribution of the number of cusps of $X_n$ (and thus that of the genus, through the Euler formula) can be understood in a similar way to the work of \textcite{Gamburd} on random Bely\u{\i} surfaces. If for instance the base surface has at least two cusps, then this setting is easier. Indeed, because the simple loops around the cusps induce primitive elements (elements that can be part of a free generating set) in the fundamental group of the base. This implies that the distributions of their images in $\mathfrak{S}_n$ are uniform, which in turn implies that in the large $n$ limit, the expected number of cusps is asymptotic to $\log(n)$ and their sizes are governed by a Poisson--Dirichlet distribution (see for instance the books by \textcite[Chapter 5]{ArratiaBarbourTavare} and \textcite[Chapter 3]{Pitman} for a definition and background).

If the base surface is closed, then the work by \textcite{MuellerPuchta1} implies that $X_n$ is connected with high probability. The length spectrum of these is also asymptotically Poisson distributed. In order to prove this, new techniques were developed by \textcite{MageePuder1,MageePuder2} and further developed by \textcite{PuderZimhoni,Maoz}.

We observe that the Laplacian spectrum of the base surface $X$ lifts to all the covers of finite degree. So the natural question to ask is what the new spectrum (see Section \ref{sec_function_spaces} for a rigorous defintion) is like. For closed base surfaces, \textcite{MageeNaudPuder} proved that, for every $\varepsilon>0$, with high probability as $n\to\infty$, the Laplacian on $X_n$ has no new spectrum below $\frac{3}{16}-\varepsilon$. So in particular, if the base surface has no spectrum below $\frac{3}{16}$ (like for instance the Bolza surface), then there is a spectral gap of at least $\frac{3}{16}-\varepsilon$. 

\textcite{HideMagee} proved that if the base surface $X$ is non-compact and of finite area, then with high probability, there is a similar relative spectral gap of $\frac{1}{4}-\varepsilon$. Since the thrice-punctured sphere has no small eigenvalues (see Lemma \ref{lem_spec_thrice_punc}), this yields a sequence of hyperbolic surfaces of finite area with a near optimal spectral gap. Using a compactification procedure that glues pairs of cusps together, due to \textcite{BuserBurgerDodziuk} (or alternatively the procedure described in Section \ref{sec_BM_model}, see the work by \textcite{BrooksMakover_ev} for details), they can be compactified so as to yield a sequence of closed surfaces with a near optimal spectral gap. \textcite{Hide2} made the $\varepsilon$ in the theorem effective, obtaining an error bound of size $O(\log\log(n)^2/\log(n))$. As we will sketch in Section \ref{sec_HM}, these results use the strong convergence result by Bordenave--Collins (Theorem \ref{thm_BordenaveCollins}) as vital input. There are by now several techniques to turn strong convergence results into near optimal spectral gaps. In the locally symmetric setting, \textcite{Magee_survey} gave a representation-theoretic proof that shows how one can turn a strong convergence statement for permutation representations of a cocompact lattice in semi-simple Lie groups into near optimal spectral gaps for finite covers of the corresponding compact orbifold.

The random cover model is also well adapted to constructing sequences of surfaces with special properties and with prescribed spectral gaps. Indeed, \textcite{LouderMagee} proved that every closed hyperbolic surface admits a sequence of random finite degree covers with a near optimal relative spectral gap as well, thus yielding arithmetic examples. \textcite{Magee_bass_notes,HidePetri} used random covers to prove that in fact all real numbers in $[0,\frac{1}{4}]$ can be approximated by bass notes of arithmetic surfaces.

The fact that uniform random finite-degree covers of a closed hyperbolic surface have a near optimal relative spectral gap was recently settled by  \textcite{MageePudervanHandel} and \textcite{HideMaceraThomas2} obtained polynomial error rates for the lower bound on spectral gap. 

\textcite{Naud2,Maoz2} proved that also for random covers, certain eigenvalue statistics follow those of large GOE and GUE matrices. \textcite{KimTao} proved a polynomial convergence rate for the density of eigenvalues in a fixed window to the density predicted by the Plancherel measure. Moreover, they proved a polynomial decay rate for the $L^p$ norms of eigenfunctions.

Finally we mention a more topological result. \textcite{KlukowskiMarkovic} proved that random covers have what they call the Putman--Wieland property with high probability as $n\to\infty$. This property was introduced by \textcite{PutmanWieland} in relation to the open question of whether mapping class groups of closed surfaces of genus at least three have virtual positive first Betti number.

\subsection{Other models}

Multiple other models for random surfaces have been studied for various purposes. \textcite{MageeNaud,CalderonMageeNaud} studied random finite covers of Schottky surfaces and proved a spectral gap for their resonances. \textcite{Mathien_diameter}, generalizing a model introduced by \textcite{BudzinskiCurienPetri_diameter}, studied surfaces obtained by randomly gluing pairs of pants together and then randomly choosing their twists and lengths. This can be thought of as a toy model for the Weil--Petersson measure normalized by the Mirzakhani function, see the works by \textcite{AranaHerrera,Liu}. \textcite{LiuPetri} defined a model of random hyperbolic surfaces with large systoles. \textcite{Moy1,HideMoyNaud,Moy2,BallmannMondalPolymerakis} determined uniform spectral gaps of random covers of surfaces of variable curvature. \textcite{MasurRafiRandecker,BowenRafiVallejos} studied the geometry of random translation surfaces of large genus. \textcite{Song,AnconaLabourieRoigSanchisToulisse} used random surfaces in spheres to show the existence of negatively curved minimal surfaces in high dimensional spheres. Finally, \textcite{Wright_covers} proved results on the geometry of random infinite degree covers of hyperbolic surfaces.

\section{A near optimal spectral gap through random covers (Hide--Magee)}\label{sec_HM}

The goal of the rest of this text is to discuss three results on near optimal spectral gaps, starting with the result by \textcite{HideMagee}. In what follows $X=\Gamma\backslash\mathbb{H}^2$ will be a fixed non-compact orientable hyperbolic surface of finite area and $X_n\to X$ will be a random $n$-sheeted cover of $X$, corresponding to a uniformly random homomorphism $\varphi_n\in\mathrm{Hom}(\Gamma,\mathfrak{S}_n)$, as described in Section \ref{sec_rand_covers}.

The spectrum of the Laplacian on $L^2(X)$ lifts to $L^2(X_n)$. Indeed, any $L^2$-function on $X$ defines an $L^2$-function on $X_n$ by composition with the covering map. Such functions are exactly the functions on $X_n$ that are constant on the fibers of the covering map. We will write $L^2_{\mathrm{new}}(X_n)$ for the orthogonal complement of these functions. So we have
\[
L^2(X_n) \simeq L^2(X) \stackrel{\perp}{\oplus} L^2_{\mathrm{new}}(X_n).
\]
The Laplacian operator preserves this decomposition, so its spectrum decomposes as
\[
\mathrm{spec}(\Delta_{X_n}) = \mathrm{spec}_{\mathrm{old}}(\Delta_{X_n}) \cup \mathrm{spec}_{\mathrm{new}}(\Delta_{X_n}),
\]
where $\mathrm{spec}_{\mathrm{old}}(\Delta_{X_n}) = \mathrm{spec}(\Delta_X)$ is the spectrum on lifts of functions on $X$ and $\mathrm{spec}_{\mathrm{new}}(\Delta_{X_n})$ is the spectrum on $L^2_{\mathrm{new}}(X_n)$. We think of all these spectra as multi-sets. That is, elements are allowed to have multiplicity and these multiplicities add when we take a union.

The main theorem of Hide and Magee's paper is:
\begin{theo}[Hide--Magee]
Let $X$ be an orientable non-compact hyperbolic surface of finite area and $X_n$ a random degree $n$ cover of $X$. Then for every $\varepsilon>0$, with high probability as $n\to\infty$, we have that
\[
\mathrm{spec}_{\mathrm{new}}(X_n) \subset \left(\frac{1}{4}-\varepsilon,\infty\right). 
\]
\end{theo}

In what follows we will sketch the proof of this theorem.

\subsection{The strategy}\label{sec_HM_strategy}

The main goal of the proof is to build, for every $s_0=s_0(\varepsilon)>\frac{1}{2}$ and every $s\in [s_0,1]$ a bounded \textbf{resolvent operator},
\[
R_{X_n}(s):L^2_{\mathrm{new}}(X_n) \to H^2_{\mathrm{new}}(X_n),
\]
where $H^2_{\mathrm{new}}(X_n)$ is a suitable Sobolev space, such that
\begin{equation}\label{eq_resolvent}
\Big(\Delta_{X_n}-s(1-s)\Big) \circ R_{X_n}(s) = \mathrm{Id}_{L^2_{\mathrm{new}}(X_n)},
\end{equation}
thus showing that $\mathrm{spec}_{\mathrm{new}}(X_n) \subset \left(s_0\cdot (1-s_0),\infty\right)$.

This resolvent is built through a \textbf{parametrix} construction. That is, instead of trying to find an exact solution to \eqref{eq_resolvent}, one looks for an operator $M_{X_n}(s):L^2_{\mathrm{new}}(X_n) \to H^2_{\mathrm{new}}(X_n)$ such that
\[
\Big(\Delta_{X_n}-s(1-s)\Big) \circ M_{X_n}(s) = \mathrm{Id}_{L^2_{\mathrm{new}}(X_n)} + E_{X_n}(s)
\]
and where $E_{X_n}(s):L^2_{\mathrm{new}}(X_n)\to L^2_{\mathrm{new}}(X_n)$ has operator norm $\lVert E_{X_n}(s) \rVert_\infty < 1$ so that  $\mathrm{Id}_{L^2_{\mathrm{new}}(X_n)} + E_{X_n}(s)$ is invertible and we can set
\[
R_{X_n}(s) = M_{X_n}(s) \circ \left(\mathrm{Id}_{L^2_{\mathrm{new}}(X_n)} + E_{X_n}(s)\right)^{-1}.
\]
The operator $M_{X_n}(s)$ will be built by ``patching together'' a parametrix in the thick part of $X_n$ and one in the thin part. For the thin part, coming from cusps, one uses a parametrix based on the resolvent on the quotient of $\mathbb{H}^2$ by a cyclic group generated by a parabolic transformation. In the thick part, the strong convergence theorem of Bordenave--Collins (Theorem \ref{thm_BordenaveCollins}) is applied. In order to make this work, the operator $E_{X_n}(s)$ will also need to be compact.

\subsection{Building the parametrix in the cusps}

As indicated above, we will build our approximate resolvent using two different procedures, depending on whether we're close to the cusps of $X_n$ or not, and patch the results together. Concretely, $M_{X_n}(s)$ will be of the form:
\[
M_{X_n}(s) = M_{X_n}^{\mathrm{int}}(s) + M_{X_n}^{\mathrm{cusp}}(s).
\]
We start with the cusp parametrix and recall some notation from Section \ref{sec_thrice_punc}. There, we defined the following cusp neighborhoods.
\[
C_t = \left\{ z\in \mathbb{H}^2;\; \mathrm{Im}(z)>t\right\} \; \Big/ \; \left\langle \left[\begin{array}{cc}
1 & 1 \\
0 & 1
\end{array}\right] \right\rangle
\]
We observe that $C_0$ is a complete hyperbolic surface of infinite area. It follows from a theorem by \textcite{Brooks_amenable} that the Laplacian on $C_0$ has no spectrum below $\frac{1}{4}$. \textcite[Lemma 4.2]{HideMagee} give a direct proof of the same fact. In particular, we have a holomorphic family of resolvents $s\mapsto R_{C_0}(s)$ for $\mathrm{Re}(s)>\frac{1}{2}$ such that
\[
\Big(\Delta_{C_0}-s(1-s)\Big) \circ R_{C_0}(s) = \mathrm{Id}_{L^2(C_0)}.
\]

The idea now is to use a cut-off function to turn these resolvents into resolvents near the cusps. The collar lemma (see the book by \textcite[Theorem 4.4.6]{Buser_book}) implies that around every cusp of the base surface $X$ we can find an isometrically embedded copy of $C_{\frac{1}{2}}$. Moreover, these cusp-neighborhoods are disjoint. We define
\[
\chi^- = \sum_{C \text{ cusp of }X}\chi^-_C \quad \text{ and }\chi^+ = \sum_{C \text{ cusp of }X}\chi^+_C 
\]
where $\chi^-_C$ and $\chi^+_C$ are smooth functions that are supported in the isometric copy of $C_1$ in $C$ and that are identically equal to $1$ on a neighborhood of $\infty$. Moreover, $\chi^-_C\cdot \chi^+_C = \chi^-_C$, i.e. the support of $\chi^-_C$ is contained in the set on which $\chi^+_C\equiv 1$. 

We now define $\chi^\pm_n = \chi^\pm \circ \pi_n$, where $\pi_n:X_n\to X$ is the covering map. This allows us to define
\[
M_{X_n}^{\mathrm{cusp}}(s): f \mapsto \chi^+_n \cdot R_{C_0}(s) \Big(\chi^-_n \cdot f\Big),
\]
There is a slight abuse of notation above: $\chi^-_C \cdot f$ can be thought of as a finite sum of functions defined on subsets of $C_0$. We may extend these functions by $0$ to functions on $C_0$ and then apply $R_{C_0}(s)$ to them. When we multiply with $\chi^+$ afterwards, we obtain functions that have a well defined extension by $0$ to $X_n$.

\textcite[(4.7) and Proposition 4.4]{HideMagee} prove that the functions $\chi_C^\pm$ can be chosen so that the following holds:
\begin{prop}\label{prp_cusp_parametrix}
Given any $s_0 > \frac{1}{2}$ there exist a choice of functions $\chi_C^\pm$ such that for any $n$ and any degree $n$ cover $\pi:Y\to X$ and any $s\in (s_0,1]$, 
\[
\Big(\Delta-s(1-s)\Big) \circ M_Y^{\mathrm{cusp}}(s) = \chi^-_Y + E^{\mathrm{cusp}}_Y(s)
\]
where $\chi_Y^- = \chi^-\circ \pi$ and $E_Y^{\mathrm{cusp}}(s)$ is a bounded operator on $L^2_{\mathrm{new}}(Y)$ whose operator norm satisfies
\[
\lVert E_Y^{\mathrm{cusp}}(s)\rVert_{L^2_{\mathrm{new}}(Y)} \leq \frac{1}{5}.
\]
\end{prop}

In the proposition above, the constant $\frac{1}{5}$ is not optimal but suffices to obtain the error bound we need below.

\subsection{Function spaces}\label{sec_function_spaces}

Before we describe the interior parametrix, we start by making the decomposition into new and old functions more explicit. First of all, writing $\mathcal{F}\subset\mathbb{H}^2$ for a closed and connected fundamental domain for $X$, we can for instance write $C^\infty_c(X_n)=C^\infty_{c,\mathrm{old}}(X_n)\oplus C^\infty_{c,\mathrm{new}}(X_n)$, where
\[
C^\infty_{c,\mathrm{old}}(X_n)
\simeq \left\{ f\in C^\infty(\mathbb{H}^2\times [n]);\; \begin{array}{c} 
f \text{ is } \Gamma \text{-invariant, }\mathrm{supp}(f)\cap(\mathcal{F}\times[n]) \text{ is compact,} \\ 
\text{and } f(x,i) = f(x,j) \text{ for all } i,j \in [n], x\in\mathbb{H}^2 
\end{array}
\right\} 
\]
and
\[
C^\infty_{c,\mathrm{new}}(X_n)\simeq 
 \left\{ f\in C^\infty(\mathbb{H}^2\times [n]);\; \begin{array}{c} 
f \text{ is } \Gamma \text{-invariant, }\mathrm{supp}(f)\cap(\mathcal{F}\times[n]) \text{ is} \\ 
\text{compact, and } \sum\limits_{i=1}^n f(x,i) = 0 \text{ for all } x\in\mathbb{H}^2
\end{array}
\right\} .
\]
The latter can also be identified with a space of vector valued functions
\[
C^\infty_{c,\mathrm{new}}(X_n)\simeq 
\left\{ f \in C^\infty(\mathbb{H}^2,\mathbb{C}_0^n);\; \begin{array}{c} 
\mathrm{supp}(f)\cap \mathcal{F} \text{ is compact, and} \\
f(\gamma \cdot x) = \rho_n(\gamma) \cdot f(x) \text{ for all } x\in\mathbb{H}^2,\; \gamma\in\Gamma
\end{array}
 \right\},
\]
where $\rho_n = \mathrm{std}_n\circ\varphi_n:\Gamma\to \mathrm{U}(\mathbb{C}_0^n)$ (for the definition of $ \mathrm{std}_n$, see Section \ref{sec_bordenavecollins}). The corresponding $L^2$-spaces $L^2_{\mathrm{new}}(X_n,\mathbb{C}_0^n)$ and Sobolev spaces $H^2_{\mathrm{new}}(X_n,\mathbb{C}_0^n)$ are the completions of these spaces with respect to the norms
\[
\lVert f \rVert_{L^2} = \int_{\mathcal{F}} \lVert f(x) \rVert^2_{\mathbb{C}^n_0} d\mu(x) \quad \text{and} \quad  \lVert f \rVert_{H^2}  := \lVert f \rVert_{L^2} + \lVert \Delta_{X_n}f \rVert_{L^2},
\]
where we have written $\mu$ for the area measure on $\mathbb{H}^2$.

\subsection{Building the parametrix away from the cusps} 

The largest part of the paper is dedicated to constructing $M_{X_n}^{\mathrm{int}}(s)$. We will sketch how this works now.

\subsubsection{Resolvents} 
We will construct $M_{X_n}^{\mathrm{int}}(s)$ using the resolvent on the hyperbolic plane. First we recall that this operator can be made very explicit. Indeed, for $\mathrm{Re}(s)>\frac{1}{2}$, the resolvent $R_{\mathbb{H}^2}(s)= (\Delta_{\mathbb{H}^2}-s(1-s))^{-1} : L^2(\mathbb{H}^2) \to L^2(\mathbb{H}^2)$ is an integral operator given by
\[
\Big(R_{\mathbb{H}^2}(s) f\Big)(x) = \int_{\mathbb{H}^2} r(s;x,y)\cdot f(y)\; d\mu(y)
\]
with
\[ 
r(s;x,y) = \frac{1}{4\pi} \int_0^1 \frac{t^{s-1}(1-t)^{s-1}}{(\mathrm{d}(x,y)-t)^s}\;dt,
\]
where $\mathrm{d}:\mathbb{H}^2\times\mathbb{H}^2 \to [0,\infty)$ denotes the hyperbolic distance function. This can for instance be found in the book by \textcite[Proposition 4.2 and Theorem 4.3]{Borthwick_hyperbolic}. In their original approach, Hide and Magee use this explicit shape in multiple computations, that we will not go over in detail here.

Because of the combination with Bordenave and Collins's theorem below, we will first need to define a cut-off. This will be the integral operator $R^{(T)}_{\mathbb{H}^2}(s)$ with kernel
\[
r^{(T)}(s;x,y) = \chi_0\Big(\mathrm{d}(x,y)-T\Big) \cdot r(s;x,y),
\]
where $\chi_0:\mathbb{R} \to [0,1]$ is a smooth function such that
\[
\chi_0(x) = 1 \text{ if }x \leq 0 \quad \text{and} \quad \chi_0(x) = 0 \text{ if }x\geq 1.
\]
In particular, if $\mathrm{d}(x,y)\geq T+1$ then $r^{(T)}(s;x,y) =0 $.

We now define the operator $K^{(T)}_{\mathbb{H}^2}(s)$ by
\[
\Big(\Delta_{\mathbb{H}^2}-s(1-s)\Big)\circ R^{(T)}_{\mathbb{H}^2}(s) = \mathrm{Id}+K^{(T)}_{\mathbb{H}^2}(s).
\]
It turns out that this is an integral operator with an explicit kernel $k^{(T)}(s;x,y)$ (see (5.8) in Hide and Magee's paper).

In order to deal with covers of degree $n$, we define integral operators $R_{\mathbb{H}^2,n}^{(T)}(s)$ and  $K_{\mathbb{H}^2,n}^{(T)}(s)$. Their action on smooth compactly supported $\mathbb{C}_0^n$-valued functions is defined by the kernels
\[
r_{\mathbb{H}^2,n}^{(T)}(s) = r_{\mathbb{H}^2}^{(T)}(s)\cdot \mathrm{Id}_{\mathbb{C}_0^n} \quad \text{and} \quad k_{\mathbb{H}^2,n}^{(T)}(s) = k_{\mathbb{H}^2}^{(T)}(s)\cdot \mathrm{Id}_{\mathbb{C}_0^n}
\]
respectively, where we have used the identifications of Section \ref{sec_function_spaces}.

Above, we defined cut-off functions $\chi^\pm_n$ in the cusps. We will write
\[
\chi^{\mathrm{int}}_n = 1-\chi^-_n
\]
and think of this as a $\Gamma$-invariant function $\mathbb{H}^2\to \mathbb{C}_0^n$. Because it's a lift of a function on the base surface $X$, all its coordinates are the same and we can think of it as a scalar valued function.

\textcite[Lemma 5.5]{HideMagee} prove that $R_{\mathbb{H},n}^{(T)}(s)\cdot \chi^{\mathrm{int}}_n$, where we have identified the function  $\chi^{\mathrm{int}}_n$ with the corresponding multiplication operator, extends to a bounded operator
\begin{equation}\label{eq_parametrix}
R_{\mathbb{H},n}^{(T)}(s)\cdot \chi^{\mathrm{int}}_n : L^2(X_n,\mathbb{C}_0^n) \to H^2(\mathbb{H}^2,\mathbb{C}_0^n)
\end{equation}
and likewise $K^{(T)}_{\mathbb{H}^2,n}\cdot \chi^{\mathrm{int}}$ extends to a bounded operator on $L^2(\mathbb{H}^2,\mathbb{C}_0^n)$ such that moreover
\[
\Big(\Delta-s(1-s)\Big) \circ R_{\mathbb{H}^2,n}^{(T)}(s)\cdot \chi^{\mathrm{int}}_n = \chi^{\mathrm{int}}_n + K_{\mathbb{H}^2,n}^{(T)}(s)\cdot \chi^{\mathrm{int}}_n.
\]
These two operators will serve as our parametrix and our error respectively. That is, we set
\[
M_{X_n}^{\mathrm{int}}(s) = R^{(T)}_{\mathbb{H}^2,n}(s)\cdot \chi^{\mathrm{int}}_n \quad \text{and} \quad E_{X_n}^{\mathrm{int}} = K^{(T)}(s)\cdot \chi^{\mathrm{int}}_n,
\]
for some $T$ that will be determined below.

\subsubsection{Random matrices}

The ultimate goal is to use the fact that the random unitary representations $\rho_n=\mathrm{std}_n\circ\varphi_n:\Gamma \to \mathrm{U}(\mathbb{C}_0^n)$ converge strongly to the left-regular representation $\rho_{\mathrm{reg.}}:\Gamma \to \ell^2(\Gamma)$. We will use a slight strengthening of Theorem \ref{thm_BordenaveCollins}. Namely, we will need to take matrix coefficients instead of scalar coefficients. That is, we need the fact that $\lVert \rho_n(z) \rVert_\infty \to \lVert \rho_{\mathrm{reg.}}(z)\rVert_\infty$ for all $z\in \mathrm{Mat}_N(\mathbb{C})\otimes_{\mathbb{C}}\mathbb{C}[\Gamma]$. In other words, we have the following consequence of Theorem \ref{thm_BordenaveCollins}:

\begin{prop}\label{prop_strong_conv}
Let $\rho_n = \mathrm{std}_n\circ \varphi_n$, where $\varphi_n \in\mathrm{Hom}(\Gamma,\mathfrak{S}_n)$ is uniformly random. Then
\[
\left\lVert\sum_{\gamma \in \Gamma} a_\gamma \otimes \rho_n(\gamma) \right\rVert_{\mathbb{C}^N\otimes \mathbb{C}^n_0} \quad \stackrel{\text{in probability}}{\longrightarrow}  \quad \left\lVert\sum_{\gamma \in \Gamma} a_\gamma \otimes \rho_{\mathrm{reg.}}(\gamma) \right\rVert_{\mathbb{C}^N\otimes \ell^2(\Gamma)} \quad \text{as }n\to\infty
\]
for all $N\in\mathbb{N}$ and all sequences $(a_\gamma)_{\gamma\in\Gamma} \in \mathrm{Mat}_N(\mathbb{C})^\Gamma$ with at most finitely many non-zero entries. 
\end{prop}

Hide and Magee use work by \textcite{Pisier} to derive this from Bordenave and Collins's result. For more on this, see the surveys by \textcite[Proposition 3.3]{Magee_survey} and \textcite[Section 2.4]{vanHandel_survey}.

\subsubsection{Applying strong convergence}

All that is left now is connecting the dots. For $f\in L^2(\mathbb{H}^2,\mathbb{C}_0^n)$ and $s\in [s_0,1]$, we have
\begin{align*}
\Big(E_{X_n}^{\mathrm{int}}(s)\cdot f\Big)(x) & = \int_{\mathbb{H}^2} k^{(T)}_{\mathbb{H}^2,n}(s;x,y) \cdot \chi^{\mathrm{int}}(y)\cdot f(y)\;d\mu(y) \\
 & = \sum_{\gamma \in \Gamma} \int_{\mathcal{F}} k^{(T)}_{\mathbb{H}^2,n}(s;\gamma\cdot x,y) \cdot \rho_n(\gamma^{-1}) \cdot \chi^{\mathrm{int}}(y)\cdot f(y) \; d\mu(y).
\end{align*}

Now the observation is that, if $\mathcal{F}$ is a fundamental domain for the action of $\Gamma$ on $\mathbb{H}^2$, then $L^2(\mathbb{H}^2,\mathbb{C}_0^n) \simeq L^2(\mathcal{F})\otimes \mathbb{C}_0^n$ as Hilbert spaces, through the map
\[
f \in L^2(\mathbb{H}^2,\mathbb{C}_0^n) \quad \mapsto \quad \sum_{i=1}^{n-1}\langle f\rvert_{\mathcal{F}},e_i\rangle \otimes e_i,
\]
where $(e_1,\ldots,e_{n-1})$ is some orthonormal basis of $\mathbb{C}_0^n$. 

This means we can think of $E_{X_n}^{\mathrm{int}}(s)$ as acting on $L^2(\mathcal{F})\otimes \mathbb{C}_0^n$ as
\[
E_{X_n}^{\mathrm{int}}(s) \simeq \sum_{\gamma\in\Gamma} a_\gamma^{(T)}(s) \otimes \rho_n(\gamma^{-1}), 
\]
where $a^{(T)}_\gamma(s):L^2(\mathcal{F}) \to L^2(\mathcal{F})$ is given by:
\[
\Big(a^{(T)}_\gamma(s) \cdot f\Big)(x) = \int_{\mathcal{F}} k_{\mathbb{H}^2}(s;\gamma\cdot x,y)\cdot \chi^{\mathrm{int}}(y)\cdot f(y)\;d\mu(y).
\]
This is already much closer to something that we can feed into Proposition \ref{prop_strong_conv}. Indeed, the cut-off using $T$ guarantees that there are only finitely many non-zero terms in the sum. 

The operators $a^{(T)}_\gamma(s)$ however aren't quite finite dimensional matrices. Nonetheless, Hide and Magee prove that they are Hilbert--Schmidt operators and hence compact. This means in particular that they can be approximated by finite dimensional matrices and this suffices to prove that, for fixed $s$ and $T$, with high probability as $n\to\infty$,
\[
\left\lVert  \sum_{\gamma\in\Gamma} a_\gamma^{(T)}(s) \otimes \rho_n(\gamma^{-1}) \right\rVert_\infty \quad \leq \quad
\left\lVert  \sum_{\gamma\in\Gamma} a_\gamma^{(T)}(s) \otimes \rho_{\mathrm{reg}}(\gamma^{-1}) \right\rVert_\infty +\frac{1}{5}.
\]
The operator $\sum_{\gamma\in\Gamma} a_\gamma^{(T)}(s) \otimes \rho_{\mathrm{reg}}(\gamma^{-1})$ is conjugated to $K^{(T)}_{\mathbb{H}^2}\cdot \chi^{\mathrm{int}}_n$. \textcite[Lemma 5.2]{HideMagee} prove that there exists a $T_0=T_0(s_0)$ such that for all $T\geq T_0$ and all $s\geq s_0$,
\[
\left\lVert K^{(T)}_{\mathbb{H}^2}(s)\cdot \chi^{\mathrm{int}}_n \right\rVert_\infty \leq \frac{1}{5}.
\]
So we obtain that with high probability \[
\lVert E_{X_n}^{\mathrm{int}}(s)\rVert_\infty \leq \frac{2}{5},
\]
for fixed $s\geq s_0$. To finish the proof, Hide and Magee show that the coefficients $a^{(T)}_\gamma(s)$ vary slowly enough as functions of $s$ to upgrade the result for fixed $s$ to one for all $s\in [0,s_0(1-s_0)]$ simultaneously, with a slight loss and thus produce the bound
\begin{equation}\label{eq_int_error_norm}
\lVert E_{X_n}^{\mathrm{int}}(s)\rVert_\infty \leq \frac{3}{5},
\end{equation}
for all $s\in(s_0,1]$ simultaneously, with high probability.

\subsection{Finishing the proof}

We now simply observe that the first part of Proposition \ref{prp_cusp_parametrix} and equation \eqref{eq_parametrix} imply that, 
\[
\Big(\Delta-s(1-s)\Big) \circ \Big(M_{X_n}^\mathrm{int}(s)+M_{X_n}^{\mathrm{cusp}}(s)\Big) = \mathrm{Id} + E_{X_n}^\mathrm{int}(s)+E_{X_n}^{\mathrm{cusp}}(s)
\]
and, by the second half of Proppsition \ref{prp_cusp_parametrix} and equation \eqref{eq_int_error_norm}, with high probability as $n\to\infty$,
\[
\lVert E_{X_n}^{\mathrm{int}}(s)+E_{X_n}^{\mathrm{cusp}}(s)\rVert_\infty \leq \frac{4}{5}
\]
This suffices to invert $\mathrm{Id} + E_{X_n}^\mathrm{int}(s)+E_{X_n}^{\mathrm{cusp}}(s)$ and hence to build our resolvent, as promised in Section \ref{sec_HM_strategy}.

\section{Weil--Petersson random surfaces (Anantharaman--Monk)}\label{sec_AM}

The next result we will treat is the theorem of \textcite{AnantharamanMonk4}. Recall from Section \ref{sec_WP} that $X_g$ denotes a random closed and orientable hyperbolic surface of genus $g$ distributed according to the Weil--Petersson measure on the moduli space $\mathcal{M}_g$.

\begin{theo}[Anantharaman--Monk]\label{thm_AM}
For every $\varepsilon>0$,
\[
\lim_{g\to\infty}\mathbb{P}\Big(\lambda_1(X_g) > \frac{1}{4}-\varepsilon\Big) = 1. 
\]
\end{theo}

The series of papers proving this theorem have a total of 312 pages on the arXiv and develop many new ideas. We have fewer pages at our disposal here, so we will only sketch some of the ingredients of Anantharaman and Monk's proof in this section. \textcite{AnantharamanMonk_overview} wrote a longer overview themselves, to which we refer the reader who would like more, but not all, details.

\subsection{The trace method}\label{sec_wp_trace_method}

The goal is to prove the theorem with an analogue in the setting of surfaces of the trace method that Friedman used (see Section \ref{sec_Friedman}). The trace of a function of the adjacency matrix is replaced by the trace of a function of the Laplacian on $X_g$, which can be accessed through the Selberg trace formula (see Section \ref{sec_STF}). We once and for all fix an even function $h\in C_c^\infty(\mathbb{R})$ with $\mathrm{supp}(h) = [-1,1]$ such that it Fourier transform $\widehat{h}$ is non-negative on $\mathbb{R}\cup i\cdot \mathbb{R}$ (because the transformation $\lambda\mapsto \sqrt{\lambda-\frac{1}{4}}$ that appears in the trace formula maps the spectrum of $\mathbb{R}\cup i[-\frac{1}{2},\frac{1}{2}]$) and define, for $R>0$, the function $h_R\in C_c^\infty(\mathbb{R})$ with support $[-R,R]$ by
\[
h_R(x) = h\left(\frac{x}{R}\right), \quad x\in\mathbb{R}.
\]
Its Fourier transform satisfies $\widehat{h_R}(\xi) = R\cdot \widehat{h}(R\cdot \xi)$, so non-negativity is preserved, and \textcite[Lemma 3.11]{AnantharamanMonk2} prove that for all $\alpha,\varepsilon>0$, there exists a constant $C_{\alpha,\varepsilon}>0$ such that for any closed hyperbolic surface $X$ and any $R\geq 1$ we have
\begin{equation}\label{eq_paleywiener}
\text{if} \quad \lambda_1(X) \leq \frac{1}{4}-\alpha^2-\varepsilon \quad \text{then} \quad \widehat{h}_R\left(\sqrt{\lambda_1(X)-\frac{1}{4}}\right) \geq C_{\alpha,\varepsilon}\cdot e^{(\alpha+\varepsilon)\cdot R}.
\end{equation}
and thus
\[
\mathbb{P}\left( \lambda_1(X_g) \leq \frac{1}{4}-\alpha^2-\varepsilon \right) \leq \mathbb{P}\left(\widehat{h}_R\left(\sqrt{\lambda_1(X_g)-\frac{1}{4}}\right) \geq C_{\alpha,\varepsilon}\cdot e^{(\alpha+\varepsilon)\cdot R} \right).
\]
Because expectations are easier to compute than probabilities, we apply Markov's inequality and obtain
\[
\mathbb{P}\left(\widehat{h}_R\left(\sqrt{\lambda_1(X_g)-\frac{1}{4}}\right) \geq C_{\alpha,\varepsilon}\cdot e^{(\alpha+\varepsilon)\cdot R} \right)  \leq \frac{\mathbb{E}\left(\widehat{h}_R\left(\sqrt{\lambda_1(X_g)-\frac{1}{4}}\right) \right)}{C_{\alpha,\varepsilon}\cdot e^{(\alpha+\varepsilon)\cdot R}} 
\]
The quantity on the right can be bounded using the Selberg trace formula. Indeed, using the fact that $\widehat{h_R}$ is non-negative on $\mathbb{R}\cup i\cdot\mathbb{R}$, all the contributions to the spectral side of the Selberg trace formula are non-negative and thus
\begin{multline*}
\mathbb{E}\left(\widehat{h}_R\left(\sqrt{\lambda_1(X_g)-\frac{1}{4}}\right) \right)
 \leq  (g-1)\cdot\int_{-\infty}^\infty  y\; \widehat{h_R}(y) \; \tanh(\pi y)\;dy 
\\  + \mathbb{E}\left(\sum_{\substack{\gamma \text{ primitive closed} \\ \text{geodesic on }X_g\\ \text{with }\ell(\gamma) \leq R}} 
\ell(\gamma) \sum_{n = 1}^{\lfloor R/\ell(\gamma) \rfloor} \frac{1}{2\sinh(n\ell(\gamma)/2)} h(n\ell(\gamma)) \right).
\end{multline*}
So if we prove that the right hand side is significantly less than $e^{(\alpha+\varepsilon)\cdot R}$, we prove that with high probability, $\lambda_1(X_g) > \frac{1}{4} - \alpha^2 -\varepsilon$.

There are however some immediate problems to overcome:
\begin{enumerate}
\item The most obvious problem is that the trace formula has a contribution coming from the trivial eigenvalue $\lambda_0(X_g)=0$. This contribution is the dominant exponential term on the spectral side and so we can never expect to prove that the right hand side is small enough.
\item A more serious problem is that the right hand side contains the integral term that has the genus in front of it. In particular, if we want this to be less than $e^{(\alpha+\varepsilon)\cdot R}$, $R$ needs to be at least $A_{\alpha,\varepsilon}\cdot \log(g)$, where $A_{\alpha,\varepsilon}\to\infty$ as $\alpha\to 0$. This means that we need to understand geodesics of (logarithmically) growing length on $X_g$. The reason that this is difficult is that geodesics of bounded length on random hyperbolic surfaces of large tend to be simple, however at logarithmic scales, they start having self-intersections and Mirzakhani's integration formula (Theorem \ref{thm_Mirzakhani_integration_formula}) does not hold for mapping class group orbits of such curves.
\item Another major (related) problem that is not yet evident from what we have derived so far is the hyperbolic analogue of the problem of tangles that appeared in Friedman's proof. Our random surface $X_g$ can have a problematic subsurface with a probability that tends to $0$ as $g\to\infty$. Just like in Friedman's proof however, these probabilities do not tend to $0$ fast enough and will ruin the expectations we're computing.
\end{enumerate}

In what follows we will sketch how Anantharaman and Monk resolve these issues, thus proving a near optimal spectral gap. 

\subsection{Removing exponential terms}\label{sec_removing_exp_terms}

First we discuss how to deal with the first problem. The first solution can be found in the articles \textcite{WuXue1,LipnowskiWright} who also used the trace method. They observed that the term corresponding to $\lambda_0(X_g)=0$ on the spectral side cancels against the contribution from simple curves on the expected geometric side and showed the remainder is small enough to prove a spectral gap of $\frac{3}{16}-\varepsilon$. The hard part of these papers is controlling the contribution from non-simple geodesics. We will not explain how this works, but will just very briefly describe how to compute the contribution coming from simple closed geodesics. We will use Anantharaman and Monk's notation:
\[
\langle F\rangle_g^{\mathbf{s}} = \mathbb{E}\left( \sum_{\substack{\gamma \text{ a simple closed} \\ \text{on }X_g}} F(\ell(\gamma)) \right).
\]
The simple closed curves on an orientable closed surface of genus $g$ fall in to $\lfloor\frac{g}{2}\rfloor+1$ mapping class group orbits. As such, it follows from Mirzakhani's integration formula (Theorem \ref{thm_Mirzakhani_integration_formula}) that
\begin{equation}\label{eq_simple_density}
\langle F\rangle_g^{\mathrm{s}} = \frac{1}{V_g}\int_0^\infty F(\ell) V_g^{\mathrm{s}}(\ell) d\ell,
\end{equation}
where $V_g^{\mathrm{s}}(\ell)$ is a function that can be expressed in terms of volumes of moduli spaces of hyperbolic surfaces with boundary. Moreover, Mirzakhani proved that 
\begin{equation}\label{eq_sinh_estimate}
\frac{V_g^{\mathbf{s}}}{V_g} = \frac{4}{\ell}\sinh^2\left(\frac{\ell}{2}\right) + O\left(\frac{(1+\ell)^c e^\ell}{g}\right),
\end{equation}
for some constant $c>0$ (this estimate has been improved by multiple authors, notably by \textcite{AnantharamanMonk1}). Putting all of the above together and canceling out the contribution from the trivial eigenvalue with the main term in $\left\langle \frac{\ell\cdot h_R(\ell)}{2\sinh(\ell/2)} \right\rangle_g^{\mathbf{s}}$ that cancel \emph{exactly}, yields
\[
\mathbb{E}\left( \widehat{h_R}\left(\sqrt{\lambda_1(X_g)-\frac{1}{4}}\right)\right) = O\left( R^2\cdot g + \frac{R^ce^{R/2}}{g}\right) + E^{\mathbf{ns}}(R,g),
\]
where $E^{\mathbf{ns}}(R,g)$ is the contribution coming from non-simple geodesics. Bounding the latter and comparing this to the exponential growth coming from the trivial eigenvalue, using \eqref{eq_paleywiener}, one finds that it's best to set $R=4\log(g)$ and this yields the lower bound $\lambda_1(X_g)\geq \frac{3}{16}-\varepsilon$.

In order to systematize these cancellations, Anantharaman and Monk replace the Fourier transform pair $(h_R,\widehat{h_R})$ by
\[
\mathcal{D}^mh_R(x)=\left(\frac{1}{4}-\frac{d^2}{dx^2}\right)^m h_R(x), \quad \left(\frac{1}{4}+\xi^2\right)^m\cdot \widehat{h_R}(\xi)
\]
for some $m\geq 1$. On the spectral side, this has the effect of removing the contribution coming from $\sqrt{\lambda_0(X_g)-\frac{1}{4}}=\frac{i}{2}$. It does mean that we can no longer exclude eigenvalues very close to $0$, but these are already covered by Mirzakhani's result (see Section \ref{sec_WP}). Filling the new function in in \eqref{eq_simple_density}, again applying \eqref{eq_sinh_estimate} and integrating by parts yields that now
\[
\left\langle \frac{\ell\cdot \mathcal{D}h_R(\ell)}{2\sinh(\ell/2)} \right\rangle_g^{\mathbf{s}}=-h_R(0)+O\left( \frac{R^ce^{R/2}}{g} \right).
\]
So the exponential term in $R$, that canceled against the contribution from the trivial eigenvalue before, now gets replaced by the bounded term $h_R(0)$. This relies crucially on the fact that $\mathcal{D}$ annihilates the function $\ell\mapsto \sinh\left(\frac{\ell}{2}\right)$.

The idea is now that if we can produce this same phenomenon for other curve types, we could produce further cancellations and obtain finer information on $\lambda_1(X_g)$.

\subsection{Friedman--Ramanujan functions}

This leads Anantharaman and Monk to the definition of what they call Friedman--Ramanujan functions. These functions are a hyperbolic analogue to the functions that Friedman calls Ramanujan functions (see Section \ref{sec_Friedman}). Before we get there, we need to talk about curves.

\begin{figure}[ht]
\begin{center}
\includegraphics{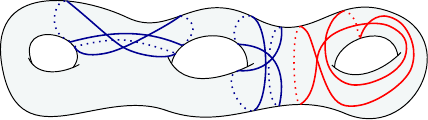}
\caption{Three curves with three self-intersections each, on a surface of genus three. The leftmost two have the same local topological type, even if they sit in the surface differently. The rightmost curve fills a one-holed torus instead of a pair of pants and thus has a different local topological type.}\label{pic_loc_type}
\end{center}
\end{figure}

To group geodesics together, Anantharaman and Monk define the \textbf{local topological type} of a geodesic. This is the data of the subsurface the geodesic fills (so a regular neighborhood of the geodesic, completed by adding in any topological disks that this regular neighborhood bounds) and the embedding of the geodesic into the surface, all considered up to homomorphism (see Figure \ref{pic_loc_type}). In particular, this data does not take into account how the subsurface that the curve fills sits in the larger surface.

For any local topological type $\mathbf{T}$, we can now define
\[
\langle F\rangle^{\mathbf{T}}_g = \mathbb{E}\left( \sum_{\substack{\gamma \text{ a closed geodesic} \\ \text{of type }\mathbf{T} \text{ on }X_g}}F(\ell(\gamma))\right).
\]
We will also write $\chi(\mathbf{T})$ for the Euler characteristic of the subsurface corresponding to $\mathbf{T}$.

In order to apply a similar strategy, we first need an analogue of \eqref{eq_simple_density} that was based on the Mirzakhani integration formula and provided us with a density for $\langle F\rangle^{\mathrm{s}}_g$. \textcite[Theorem 5.15]{AnantharamanMonk2} prove that for any topological type $\mathbf{T}$ there exists a unique sequence of continuous functions $(f^{\mathbf{T}}_k)_{k\geq \lvert\chi(\mathbf{T})\rvert}$ such that
\begin{equation}\label{eq_density_general_types}
\langle F\rangle^{\mathbf{T}}_g = \sum_{k=\lvert \chi(\mathbf{T})\rvert}^K \frac{1}{g^k} \int_0^\infty F(\ell)\cdot f^{\mathbf{T}}_k(\ell)\; d\ell + O_{\mathbf{T},K,\eta}\left( \frac{\left\lVert F(\ell) e^{(1+\eta)\ell}\right\rVert_\infty}{g^{K+1}}\right)
\end{equation}
for any $\eta>0$, any $K\geq 0$ and any $g$ large enough.

In order to apply an analogous strategy to that of integration by parts and the fact that $\mathcal{D}$ annihilates the function in Mirzakhani's $\sinh$ estimate \eqref{eq_sinh_estimate}, we need to know that $\mathcal{D}^m$ ``almost'' annihilates the coefficient functions $f_k^{\mathbf{T}}$. This is what Friedman--Ramanujan functions are for.

\begin{defi}
Let $c\geq 0$. We define
\[
\mathcal{R}^c_w = \left\{f\in C^0(\mathbb{R}_{\geq 0},\mathbb{C});\; \exists c_1>0\text{ such that }\forall n\geq 1:\; \int_0^n\lvert f(s)\rvert ds \leq c_1 (n+1)^c e^{n/2}\right\}.
\]
We also set $\mathcal{R}_w = \cup_{c\geq 0}\mathcal{R}^c_w$. A function $f\in C^0(\mathbb{R}_{\geq 0},\mathbb{C})$ is said to be \textbf{Friedman--Ramanujan} if there exists a polynomial $p$ such that
\[
f(\ell) -p(\ell)\cdot e^\ell \in \mathcal{R}_w.
\]
The function $\ell\mapsto p(\ell)\cdot e^\ell $ is called the \textbf{principal term} of $f$ and $f(\ell)-p(\ell)\cdot e^\ell$ is called the remainder term of $f$.
\end{defi}

The main technical result of the final article of \textcite[Theorem 1.15]{AnantharamanMonk4} is: 
\begin{theo}[Anantharaman--Monk]
The functions 
\[
\ell \mapsto \left\{
\begin{array}{ll}
\ell \cdot f_k^{\mathbf{T}}(\ell) & \text{if } \mathbf{T}=\mathbf{s} \\
f_k^{\mathbf{T}}(\ell) & \text{otherwise},
\end{array}
\right.
\]
in \eqref{eq_density_general_types} are Friedman--Ramanujan.
\end{theo}

The proof of this theorem forms a large portion of the total and develops many new ideas, like for instance new coordinates adapted to the local topological type $\mathbf{T}$, that replace the Fenchel--Nielsen coordinates. We will not go into this here and refer to the paper by Anantharaman and Monk for details.

Using the fact that $\mathcal{D}^m$ annihilates functions of the form $\ell\mapsto p(\ell)\cdot e^{\ell/2}$ with $p$ a polynomial of degree $\leq m-1$ allows for an argument using integration by parts again and yields that
\[
\left\langle \ell e^{-\ell/2}\mathcal{D}^m f_R(\ell) \right\rangle^{\mathbf{T}}_g = O_{m,\mathbf{T},K,\eta}\left(R^{c(K,\mathbf{T})} + \frac{e^{(\frac{1}{2}+\eta)\cdot R}}{g^{K+1}}\right).
\]
So the contribution of $\mathbf{T}$ is only polynomial as a function of $R$, instead of exponential. Balancing the error term above with the topological term in the trace formula, that is linear in $g$, leads one to choose $R=2(K+2)\log(g)$. This yields
\[
\left\langle \ell e^{-\ell/2}\mathcal{D}^m f_R(\ell) \right\rangle^{\mathbf{T}}_g = o\left(e^{(\alpha+\varepsilon)R}\right)
\]
with $\alpha = \frac{1}{2(K+2)}$, which is still the same $\alpha$ that appears in the lower bound on the spectral gap. In other words, the larger we can make $K$, the better the bound on the spectral gap we get.

\subsection{Möbius inversion and the problem of tangles}

There is one major problem left: we still need to sum over all possible local topological types of geodesics of length at most $R=2(K+2)\log(g)$. This is a serious problem, because, as \textcite[Theorem 9.1]{AnantharamanMonk2} prove, already the function 
\[
\ell\mapsto \ell\cdot f_1^{\mathbf{all}}(\ell) = \ell\cdot \sum_{\mathbf{T}} f_1^{\mathbf{T}}(\ell),
\]
where the sum runs over all (countably many) local topological types $\mathbf{T}$, is \emph{not} a Friedman--Ramanujan function. In short, the problem is that the number of local topological types that we need to sum over grows exponentially as a function of our parameter $R$. So, a priori, the strategy completely falls apart.

As we have already mentioned, the problem is the presence of tangles. Here, for $\kappa,\omega>0$ with $\kappa<\omega$, a \textbf{$(\kappa,\omega)$-tangle} is one of the following:
\begin{itemize}
\item A Riemannian circle $\gamma$ of length $\ell(\gamma)\leq \kappa$, or
\item A pair of pants or a one-holed torus whose longest boundary component has length at most $\omega$.
\end{itemize}
A $(\kappa,\omega)$-tangle in a hyperbolic surface $X$ is a totally geodesically embedded $(\kappa,\omega)$-tangle in $X$, so either a compact subsurface of Euler characteristic $-1$ with totally geodesic boundary or a simple closed geodesic. A surface that has no $(\kappa,\omega)$-tangles in it is called \textbf{$(\kappa,\omega)$-tangle free}.

\textcite[Theorem 4.6, Corollary 4.7 and Corollary 4.8]{AnantharamanMonk3} prove a vast generalization of the collar lemma for tangle-free surfaces: for $0<\kappa<1$, $\omega = \kappa\log(g)$, $A\geq 1$, $R=A\cdot \log(g)$ and $\chi_{\mathrm{min}} \leq -1$, there exists a constant $\beta_{\kappa,A,\chi_{\mathrm{min}}}>0$ such that a $(\kappa,\omega)$-tangle free surface $X\in\mathcal{M}_g$ supports at most
\[
O_{\kappa,A,\chi_{\mathrm{min}}}\left( (\log(g))^{\beta_{\kappa,A,\chi_{\mathrm{min}}}}\right)
\]
local topological types $\mathbf{T}$ of closed geodesics $\gamma$ with $\chi(\mathbf{T}) \geq \chi_{\mathrm{min}}$ and $\ell(\gamma)\leq R$. In other words, under the tangle free hypothesis, we go from exponentially many local topological types we need to deal with to only polynomially many.

So, if we were able to condition on the set of tangle-free surfaces, we would win. The solution is to use Möbius inversion (a version of the inclusion-exclusion principle). The basic idea is that, if $N:\Omega\to\mathbb{N}$ is a counting variable on some probability space $\Omega$, then the indicator function $\mathds{1}_{\{N=0\}}$ of the event that $N=0$ can be written as
\[
\mathds{1}_{\{N=0\}} = 1-\sum_{k=1}^\infty \frac{(-1)^k}{k!} \cdot N(N-1)\cdots (N-k+1).
\]
Moreover, the variable $N(N-1)\cdots (N-k+1):\Omega\to \mathbb{N}$ is a counting variable too, it counts the number of ordered $k$-tuples of whatever $N$ counts.

So, we would want to set $N=T^{\kappa,\omega}:\mathcal{M}_g\to \mathbb{N}$ in the above, where $T^{\kappa,\omega}(X)$ counts the number of $(\kappa,\omega)$-tangles in $X$. This means that $\mathds{1}_{\{T^{\kappa,\omega}=0\}}$ is exactly the indicator of the event that $X$ is $(\kappa,\omega)$-tangle free. However, tangles can intersect in all sorts of complicated ways, which makes controlling the expectation of the right hand side of the inclusion-exclusion formula difficult. So, to solve this issue, \textcite[Theorem 3.1]{AnantharamanMonk3} write down a different, less explicit, Möbius function. We refer to their article for details. 

In short, putting all of the above together produces the necessary cancellations to make the trace method work and this proves Theorem \ref{thm_AM}.

\section{Polynomial error bounds (Hide--Macera--Thomas)}\label{sec_HMT}

The final result we will discuss is the recent theorem by \textcite{HideMaceraThomas}. They prove, for a Weil--Petersson distributed random surface $X_g$:

\begin{theo}\label{thm_HMT}
There exists a constant $c>0$ such that
\[
\lim_{g\to\infty}
\mathbb{P}\left(\lambda_1(X_g) > \frac{1}{4}-\frac{1}{g^c} \right) = 1.
\]
\end{theo}

This theorem in particular recovers Theorem \ref{thm_AM} by Anantharaman--Monk. The proof starts off in the same way, namely by writing down an asymptotic expansion in powers of $g^{-1}$ for the expected trace of a function of the Laplacian. However, the way this expression is analyzed is fundamentally different. The method is heavily inspired by the extension of the polynomial method by \textcite{MageePudervanHandel}. This does not recover the many intermediate results of Anantharaman and Monk that are of independent interest, but it does make for a (still long, but) shorter proof and improves the estimate on $\lambda_1(X_g)$.

\subsection{The interest in polynomial error bounds}\label{sec_why_poly}

Let us first explain why one might care about polynomial error terms. The ultimate goal is to understand the distribution of $\lambda_1(X_g)-\frac{1}{4}$ to a sufficient degree of precision so as to obtain the result that for all $g$ large enough,
\[
\mathbb{P}\left(\lambda_1(X_g) > \frac{1}{4} \right) >0,
\]
thus proving the existence of surfaces of arbitrarily large genus with a spectral gap that is larger than that of the hyperbolic plane. As we've mentioned above, \textcite{HuangMcKenzieYau} have successfully carried out the graph theoretic analogue of this strategy.

Of course, by Theorem \ref{thm_AM} and \eqref{eq_Huber_bound}, $\lambda_1(X_g)-\frac{1}{4}\to 0$ in probability. So the real question is how fast it tends to $0$ as a function of $g$. Once we know this, we can rescale the random variable and hope to understand the distribution. Looking at Section \ref{sec_BS_convergence}, we expect $X_g$ to have roughly 
\[
(g-1) \cdot \int_0^\varepsilon \tanh(\pi\sqrt{y}) dy \stackrel{\varepsilon\to 0}{=} \frac{2\pi}{3}\cdot (g-1) \cdot \left( \varepsilon^{3/2} + O\left(\varepsilon^{5/2}\right)\right)
\]
eigenvalues in the interval $[\frac{1}{4},\frac{1}{4}+\varepsilon]$. So, reverse engineering this, we get that if we want an interval with only a constant number of eigenvalues in it, we need to take $\varepsilon =\varepsilon(g) = (g-1)^{-2/3}$. So we would hope that, as $g\to\infty$, $(g-1)^{2/3}\cdot \left(\lambda_1(X_g)-\frac{1}{4}\right)$ converges to a random variable we can understand. This is exactly what happens in the graph-theoretic setting: at exactly the scale $n^{2/3}$ (where $n$ is the number of vertices) the distribution converges to a Tracy--Widom distribution. In short, one hopes that Theorem \ref{thm_HMT} holds for $c=\frac{2}{3}-\varepsilon$ for all $\varepsilon>0$.

\subsection{An asymptotic expansion}

Now we start explaining the strategy of proof. Again, we once and for all fix an even, non-negative function $h\in C_c^\infty(\mathbb{R})$. The main technical result in the paper by \textcite[Theorem 1.4]{HideMaceraThomas}, in which $h^{\ast m}$ means the $m$-fold convolution of the function $h$, is:
\begin{theo}[Hide--Macera--Thomas]\label{thm_HTM_expansion}
There exists a constant $c>0$ such that for all $m\in \mathbb{N}$ there exists a sequence of constants $\Big(a_k^{(m)}\Big)_{k\geq 0}$ such that
\[
\left\lvert \mathbb{E}\left( \frac{1}{g}\cdot \mathrm{Tr}\left( \widehat{h}\left(\sqrt{\Delta_{X_g}-\frac{1}{4}}\right)^m  \right)\right) - \sum_{k=0}^{q-1} \frac{a^{(m)}_k}{g^k} \right\rvert \leq \frac{(cq)^{cq}}{g^q}
\]
for all $q>m$ and $g>cq^c$. Moreover
\[
a_0^{(m)} = \int_{\frac{1}{4}}^\infty \widehat{h}\left(\sqrt{y-\frac{1}{4}}\right)^m \tanh\left(\pi \sqrt{y-\frac{1}{4}}\right)\; dy
\]
and
\[
a_1^{(m)} = \int_0^\infty \sum_{k=1}^\infty 2 \frac{\sinh\left( \frac{\ell}{2} \right)^2}{\sinh\left(\frac{k\ell}{2}\right)} h^{\ast m}(k\ell)\; d\ell -  a_0^{(m)}.
\]
\end{theo}

At first sight, this theorem looks very much like some of the things we have already seen. Indeed, \textcite{AnantharamanMonk2} had already observed,
using the volume estimates due to \textcite{MirzakhaniZograf,AnantharamanMonk1}, that an asymptotic expansion for expectations like the one above exists (see e.g. \eqref{eq_density_general_types} above). Moreover, as Hide, Macera and Thomas note, the fact that the coefficients $a_0^{(m)}$ and $a_1^{(m)}$ take the stated shape can be derived using computations very similar to those performed by \textcite{MirzakhaniPetri}.

The main point of the theorem above is the innocent looking error estimate. The largest part of the work in the paper goes into proving this estimate, namely, that the constant in the ``$O(g^{-q})$'' estimate grows at most factorially fast as a function of $q$. To prove this, Hide, Macera and Thomas analyze the known recurrences for Weil--Petersson volumes and intersection numbers of tautological classes on moduli spaces and need to make many estimates effective. This is difficult and we will not describe this in detail here. We refer to their paper for details.

The upshot of all of this is that, once the theorem above has been established, unlike in the previous section, we don't need to know anything further about the remaining coefficients except their existence. We also no longer need to analyze geodesics whose lengths grow as a function of the genus. Instead, one can adapt the strategy that \textcite{MageePudervanHandel} used in the proof of a near optimal spectral gap for random finite degree covers of a closed surface, based on the analogue of the theorem above in that context (Theorem 1.7 in the paper by Magee, Puder and van Handel).

\subsection{Deducing a spectral gap}

Just like in the previous section, we will only exclude eigenvalues in an interval $\left(\delta,\frac{1}{4}-g^{-c}\right)$ for a fixed $c>0$ and some small $\delta>0$ and rely on for instance Mirzakhani's result (see Section \ref{sec_WP}) to exclude eigenvalues close to $0$. The number $\delta$ will be a constant in what follows, so given Mirzakhani's bound, we can for instance choose $\delta=0.002$. This assumption of an a priori spectral gap is not strictly necessary, because a similar cancellation to the one observed by Wu--Xue and Lipnowski--Wright (see Section \ref{sec_removing_exp_terms} above) can also be used to deal with the spectrum close to $0$. We will not provide further details here.

The first goal of the remainder of the proof is to extend Theorem \ref{thm_HTM_expansion} from polynomials in $\widehat{h}\left(\sqrt{\Delta_{X_g}-\frac{1}{4}}\right)$ to smooth functions of the same quantity. There is an immediate problem, namely that for an arbitrary smooth function $f$, the operator $f\left(\widehat{h}\left(\sqrt{\Delta_{X_g}-\frac{1}{4}}\right)\right)$ is not trace class. To solve this, Hide, Macera and Thomas consider only functions of the form 
\[
x\mapsto f(x) = x \cdot \widetilde{f}(x),
\]
where $\widetilde{f} \in C^\infty(\mathbb{R})$ is any function.

First of all, we turn the coefficients $a_0^{(m)}$ and $a_1^{(m)}$ into linear functionals on polynomials. Because of the extra multiplication by $x$, we will input a polynomial $p(x) \in \mathbb{C}[x]$, but compute the coefficients for $x\cdot p(x)$. That is, we set
\[
\nu_0\Big( \sum_{m=0}^{q-1} s_m x^m \Big) = \sum_{m=0}^{q-1} s_m \cdot \int_{\frac{1}{4}}^\infty \widehat{h}\left(\sqrt{y-\frac{1}{4}}\right)^{m+1} \tanh\left(\sqrt{y-\frac{1}{4}} \right)\;dy
\]
and
\[
\nu_1\Big( \sum_{m=0}^{q-1} s_m x^m \Big) = \sum_{m=0}^{q-1} s_m\cdot \int_0^\infty \sum_{k=1}^\infty 2 \frac{\sinh\left(\frac{\ell}{2}\right)^2}{\sinh\left(\frac{k\ell}{2}\right)}h^{\ast (m+1)}(k\ell) \; d\ell  -\nu_0\Big( \sum_{m=0}^{q-1} s_m x^m \Big)
\]
\textcite[Proposition 3.2]{HideMaceraThomas} then use Theorem \ref{thm_HTM_expansion} to prove, following a similar strategy to the papers by \textcite{ChenGarzaVargasTroppvanHandel,MageePudervanHandel}, based on a variant of the Markov brothers' inequality, that these functionals extend to compactly supported distributions. That is, we can evaluate them on smooth functions $\widetilde{f}\in C^\infty(\mathbb{R})$ and we moreover have
\begin{multline}\label{eq_expectation_bound}
\left\lvert \mathbb{E}\left( \frac{1}{g}\cdot \mathrm{Tr}\left( f\left(\widehat{h}\left(\sqrt{\Delta_{X_g}-\frac{1}{4}}\right)\right)\right)\right) - \nu_0(\widetilde{f}) - \frac{\nu_1(\widetilde{f})}{g}\right\rvert \\
\leq \frac{C}{g^2}\cdot \left( \lVert w^{(m)} \rVert_{[0,2\pi]} + \lVert \widetilde{f} \rVert_{\left[-\widehat{h}\left(\frac{i}{2}\right),\widehat{h}\left(\frac{i}{2}\right)\right]}\right)
\end{multline}
for some uniform $C,m>0$ for all $\widetilde{f}\in C^\infty(\mathbb{R})$ and all $g\geq 2$. We have again set $f(x) = x\cdot f(x)$ and the function $w\in C^\infty([0,2\pi])$ is defined by $w(\theta) = \widetilde{f}\Big( \widehat{h}\left(\frac{i}{2}\right)\cdot \cos(\theta)\Big)$ and $w^{(m)}$ is its $m$\textsuperscript{th} derivative with respect to $\theta$. Finally, $\lVert \;\cdot\; \rVert_{[a,b]}$ denotes the $\sup$-norm restricted to $[a,b]$.

Next, \textcite[Lemma 3.3.]{HideMaceraThomas} investigate the support of $\nu_0$ and $\nu_1$. This is where the lower bound $\delta>0$ comes into play. Using the fact that $\nu_0$ and $\nu_1$ have a very explicit form when evaluated on polynomial, one proves that, if $f\in C^\infty(\mathbb{R})$ has support 
\[
\mathrm{supp}(f) \subset \left(\widehat{h}(0),\widehat{h}\left(i\cdot\sqrt{\frac{1}{4}-\delta}\right)\right)
\]
then for $\widetilde{f}(x) = f(x)/x$ we have
\[
\nu_0(\widetilde{f}) = \nu_1(\widetilde{f}) = 0.
\]

Using a lemma from \textcite[Lemma 4.10]{ChenGarzaVargasTroppvanHandel}, Hide Macera and Thomas prove that, for any $\eta>0$ small enough, one can find a non-negative function $f\in C^\infty(\mathbb{R})$ such that 
\[
f(x) = 1 \quad \text{for all } x \in \left[ \widehat{h}\left(i\cdot \sqrt{\varepsilon}\right), \widehat{h}\left(i\cdot \sqrt{\frac{1}{4}-\delta-\eta}\right)\right],
\]
\[
f(x) = 0 \quad \text{for all } x \in (-\infty,\widehat{h}(0)] \; \cup \; \left[\widehat{h}\left(i\cdot \sqrt{\frac{1}{4}-\delta+\eta}\right),\widehat{h}\left(\frac{i}{2}\right)\right]
\]
and moreover
\[
\lVert w^{(m)} \rVert_{[0,2\pi]} \leq c(m)\cdot \varepsilon^{-m/2}
\]
for all $\varepsilon>0$ small enough.

This means that we obtain
\begin{align*}
\mathbb{P}\left( \delta+\eta \leq \lambda_1(X_g) \leq \frac{1}{4}-\varepsilon\right) & \leq \mathbb{P}\left(  \mathrm{Tr}\left( f\left(\widehat{h}\left(\sqrt{\Delta_{X_g}-\frac{1}{4}}\right)\right) \right) \geq 1\right) \\
& \stackrel{\text{Markov}}{\leq} g\cdot \mathbb{E}\left( \frac{1}{g} \mathrm{Tr}\left( f\left(\widehat{h}\left(\sqrt{\Delta_{X_g}-\frac{1}{4}}\right)\right) \right) \right).
\end{align*}
Now we use \eqref{eq_expectation_bound} and obtain that
\[
\mathbb{P}\left( \delta+\eta \leq \lambda_1(X_g) \leq \frac{1}{4}-\varepsilon\right) \leq C\cdot \frac{\varepsilon^{-m/2}}{g},
\]
which tends to $0$ as long as $\varepsilon = \varepsilon(g) = o\left(g^{2/m}\right)$ as $g\to\infty$ and this proves the main result of this section.

\newpage



\printbibliography

\end{document}
